\documentclass[12pt]{amsart}

\usepackage{graphicx}

\DeclareMathAlphabet{\mathpzc}{OT1}{pzc}{m}{it}

\newtheorem*{thm}{Theorem}
\newtheorem{theorem}{Theorem}

\newtheorem{lemma}{Lemma}
\newtheorem{corollary}{Corollary}

\newtheorem*{thm2'}{Theorem 2$^\prime$}

\newcommand{\Sch}{\mathcal{S}}
\newcommand{\U}{\mathcal{U}}
\newcommand{\E}{\mathcal{E}}
\newcommand{\Mf}{\mathcal{M}\tilde{f}}
\newcommand{\RR}{\mathcal{R}}

\newcommand{\z}{\zeta}

\newcommand{\dl}{\delta}

\newcommand{\sg}{\sigma}

\newcommand{\tw}{{\tilde{w}}}
\newcommand{\tf}{\tilde{f}}

\newcommand{\D}{\mathbb{D}}
\newcommand{\C}{\mathbb{C}}
\newcommand{\R}{\mathbb{R}^3}
\newcommand{\Hyp}{\mathbb{H}^3}
\newcommand{\Sph}{\mathbb{S}}

\renewcommand{\Re}{\operatorname{Re}}
\renewcommand{\Im}{\operatorname{Im}}

\title{Quasiconformal Extensions to Space of Weierstrass-Enneper Lifts}

\thanks{The authors were supported in part by FONDECYT Grant \# 1110321}

\author{M. Chuaqui \and P. Duren \and B. Osgood}

\address{P. Universidad Cat\'olica
de Chile}
\email{mchuaqui@mat.puc.cl}
\address{University of Michigan}
\email{duren@umich.edu}
\address{Stanford University}
\email{osgood@stanford.edu}



\date{}

\dedicatory{To the memory of Professor F.W. Gehring}

\begin{document}

\maketitle

\begin{abstract}
The Ahlfors-Weill extension of a conformal mapping of the disk is generalized to the Weierstrass-Enneper lift of a harmonic mapping of the disk to a minimal surface, producing homeomorphic and quasiconformal extensions to space. The extension is defined through the family of best M\"obius approximations to the lift applied to a bundle of Euclidean circles orthogonal to the disk. Extension of the planar harmonic map is also obtained subject to additional assumptions on the dilatation. The hypotheses involve bounds on a generalized Schwarzian derivative for harmonic mappings in terms of the hyperbolic metric of the disk and the Gaussian curvature of the minimal surface. Hyperbolic convexity plays a crucial role.

\end{abstract}

\section{Introduction}

 If $f$ is an analytic, locally injective function its Schwarzian derivative is
\[
Sf = \left(\frac{f''}{f'}\right)'-\frac{1}{2}\left(\frac{f''}{f'}\right)^2.
\]
We owe to Nehari \cite{nehari:schlicht} the discovery that the size of the Schwarzian derivative of $f$ is related to its injectivity, and to Ahlfors and Weill \cite{ahlfors-weill:extension} the discovery of an  allied, stronger phenomenon of quasiconformal extension. We  state the combined results as follows:
\begin{theorem}[Nehari, Ahlfors-Weill] \label{theorem:plane-case}
Let $f$ be analytic and locally injective in the unit disk, $\D$.
\begin{list}{}{\setlength\leftmargin{.25in}}
\item[(a)] If
\begin{equation} \label{eq:nehari-2}
|Sf(z)| \le \frac{2}{(1-|z|^2)^2}, \quad z \in \D,
\end{equation}
then $f$ is injective in $\D$.
\item[(b)] If for some $\rho<1$,
\begin{equation} \label{eq:aw-2t}
|Sf(z)| \le\frac{2\rho}{(1-|z|^2)^2}, \quad z \in \D,
\end{equation}
then $f$ has a $\frac{1+\rho}{1-\rho}$-quasiconformal extension to ${\mathbb{C}}$.
\end{list}
\end{theorem}
A remarkable aspect of Ahlfors and Weill's theorem is the explicit formula they give for the  extension. They need the stronger inequality \eqref{eq:aw-2t} to show, first of all, that  the extended mapping has a positive Jacobian and is hence a local homeomorphism. Global injectivity then follows from the monodromy theorem and quasiconformality from a calculation of the dilatation. The topological argument cannot get started without \eqref{eq:aw-2t}, but a different approach in \cite{co:aw} shows that the same Ahlfors-Weill formula still provides a homeomorphic extension even when $f$ satisfies the weaker inequality \eqref{eq:nehari-2} and   $f(\D)$ is a Jordan domain. As to the latter requirement, if $f$ satisfies \eqref{eq:nehari-2} then $f(\D)$ fails to be a Jordan domain only when $f(\D)$ is a parallel strip or the image of a parallel strip under a M\"obius transformation, as shown by Gehring and Pommerenke \cite{gp:nehari}.

 \smallskip

In earlier work, \cite{cdo:harmonic-schwarzian}, \cite{cdo:injective-lift}, we introduced a Schwarzian derivative for harmonic mappings in the plane and we established an injectivity criterion analogous to \eqref{eq:nehari-2} for the Weierstrass-Enneper lift of a harmonic mapping of $\D$  to a minimal surface. For  a very interesting generalization we also call attention to the important paper of D. Stowe, \cite{raz:Ahlfors-derivative}. In this paper we show that homeomorphic and quasiconformal extensions also obtain in this more general setting under a hypothesis analogous to \eqref{eq:aw-2t}.  The construction is a geometric generalization of the Ahlfors-Weill formula and extends the lift not just to the plane but to all of space.

To state our results we need some terminology and notation for harmonic mappings; we refer to \cite{duren:harmonic} for more details. Let $\D$ denote the unit disk in the complex plane and let $f \colon \D \to \mathbb{C}$ be harmonic. As is customary, we write $f=h+\bar{g}$ where $g$ and $h$ are analytic. If $|h'|+|g'| \ne 0$ and the dilatation $\omega = g'/h'$ is the square of a meromorphic function, then there is a lift $\tw=\tf(z)$ of $f$ mapping $\D$ onto a minimal surface $\Sigma$ in $\R$. The function $\tf$ is called the Weierstrass-Enneper parametrization of $\Sigma$. Its three components are themselves harmonic functions and $\tf$ is a conformal mapping of $\D$ onto $\Sigma$ with conformal metric
\[
\tf^*(|d\tw|^2)=e^{2\sigma(z)}|dz|^2, \quad  e^\sigma = |h'|+|g'|,
\]
on $\D$. Then
\[
\langle \partial_x\tf,\partial_x \tf\rangle = \langle \partial_y \tf,\partial_y\tf\rangle = e^{2\sigma}, \quad \langle \partial_x \tf,\partial_y \tf\rangle = 0, \quad z =x+iy,
\]
where $\langle\cdot,\cdot\rangle$ denotes the Euclidean inner product. The Gaussian curvature of $\Sigma$ at a point $\tf(z)$ is
\[
K(\tf(z)) = -e^{-2\sg(z)}\Delta \sg(z).
\]

As introduced in  \cite{cdo:harmonic-schwarzian}, the Schwarzian derivative of $\tf$  is
\begin{equation} \label{eq:harmonic-schwarzian}
\Sch \tf =   2(\partial_{zz}\sg-(\partial_z\sg)^2).  
\end{equation}
This becomes the familiar Schwarzian when $\tf$ is analytic and $\Sigma \subset \mathbb{C}$, where then $\sg=\log|\tf'|$.

The principal result of this paper is the following generalization of the Ahlfors-Weill theorem.

\begin{theorem} \label{theorem:extension} Let $0 \le \rho \le 1$. Suppose $\tf$ satisfies
\begin{equation} \label{eq:AW-condition}
|\Sch \tf(z)| + e^{2\sigma(z)}|K(\tf(z))| \le \frac{2\rho}{(1-|z|^2)^2},\quad z \in \D.
\end{equation}
Then $\tf$ is injective. If $\rho<1$ then $\tf$ has a  $k(\rho)$-quasiconformal extension $\E\tf$ to $\overline{\R}$. If $\rho=1$ and $\partial\Sigma$ is a Jordan curve then $\E\tf$ is a homeomorphism.
\end{theorem}

 That $\tf$ is injective in $\D$  was proved in \cite{cdo:injective-lift}  in even greater generality, so the point here is the extension.   It was also proved in \cite{cdo:injective-lift} that if $\tf$ satisfies \eqref{eq:AW-condition} with $\rho=1$ then $f$ and $\tf$ have spherically continuous extensions to $\partial\D$. Furthermore, we know exactly when $\partial\Sigma$ fails to be a Jordan curve in $\mathbb{R}^3$, namely when either $\tf$ is analytic and $\tf(\D)$ is the M\"{o}bius image of a parallel strip, or when  $\tf$ maps $\overline{\D}$ into a catenoid  and $\partial\Sigma$ is pinched by a Euclidean circle on the surface. In either case, there is a Euclidean circle $C$ on $\overline{\Sigma}$ and a point $p\in C$ with $\tf(\zeta_1) = p=\tf(\zeta_2)$ for a pair of points $\zeta_1,\zeta_2 \in \partial\D$. Furthermore,  equality holds in \eqref{eq:AW-condition} with $\rho=1$ along $\tf^{-1}(C \setminus\{p\})$, and because of this a function satisfying the stronger inequality with $\rho<1$ or the strict inequality with $\rho=1$ is always injective on $\partial\D$.

  It follows from properties of the Schwarzian and from Schwarz's lemma that if $\tf$ satisfies \eqref{eq:AW-condition} then so does $\tf \circ M$ for any M\"obius transformation $M$ of $\D$ onto itself. Note also that the condition trivially entails a bound on the curvature,
 \begin{equation} \label{eq:curvature-bound}
 e^{2\sigma(z)}|K(\tf(z))| \le \frac{2\rho}{(1-|z|^2)^2}.
 \end{equation}

The extension $\E\tf$  is defined in equation \eqref{eq:Ef} in Section \ref{section:convexity}. It is constructed by setting up a correspondence between two fibrations of space by Euclidean circles, one based on $\D$ and the other on $\Sigma$. Fundamental properties of these fibrations rely
on the convexity relative to the hyperbolic metric of a real-valued function, denoted $\U\tf$ and defined in Section \ref{section:convexity-U-S1f}, naturally associated with conformal mappings of $\D$ into $\R$; it is the pullback under $\tf$ of the square root of what can naturally be regarded as the Poincar\'e metric of $\Sigma$. The arguments rely on comparison theorems for differential equations. Of particular interest is the use of a Schwarzian derivative for curves introduced by Ahlfors, \cite{ahlfors:schwarzian-R^n}.

The correspondence between the two fibrations that defines $\E\tf$ is via $p \mapsto \Mf(p,\zeta)$, $p\in \R$, using the family $\Mf(p,\zeta)$ of  {best M\"{o}bius approximations} to $\tilde{f}$ parametrized by $\zeta\in \D$.  Sections \ref{section:BMA-1} and \ref{section:BMA-2} study  best M\"obius approximations in some detail and provide formulas and properties that underly  the proof of quasiconformality of $\E\tf$ in Section  \ref{section:Ef-quasiconformal}. Moreover, in Section \ref{section:BMA-1} we show that restricting the extension $\E\tf$ to $\mathbb{C}$ yields a reflection $\RR$ across $\partial\Sigma$ with a formula quite like the Ahlfors-Weill reflection. In particular, $\RR$ sews the reflected surface, $\Sigma^*=\RR(\Sigma)$, onto $\Sigma$ along the boundary. Then the topological sphere $\overline{\Sigma}\cup\Sigma^*$ is a quasisphere, being the image of $\overline{\mathbb{C}}$ under the quasiconformal mapping $\E\tf$ of $\overline{\R}$, and $\partial\Sigma$ is a spatial quasicircle, being the image of $\partial\D$.  When $f$ is analytic all aspects of the construction and the theorem reduce to the classical results, including the bound for the quasiconformality of $\E\tf$ which becomes $k(\rho) = (1+\rho)/(1-\rho)$.

Our analysis of the quasiconformality of $\E\tf$ is very much influenced by C. Epstein's insightful treatment of the classical theorems  in \cite{ epstein:reflections}, which relies on aspects of hyperbolic geometry of the upper half-space and parallel flow. However, as will be apparent, the nonzero curvature of $\Sigma$ is a considerable complication and a new approach is necessary.

As a corollary of this work, in Section \ref{section:planar-extension} we will derive a sufficient condition for quasiconformal extension
of planar harmonic mappings $f=h+\bar{g}$. This is perhaps closer to the original Ahlfors-Weill result in that we obtain simultaneously an injectivity
criterion for harmonic mappings together with a quasiconformal extension.

\begin{theorem} \label{theorem:planar-extension}
Suppose $f= h+\bar{g}$ is a locally injective harmonic mapping of $\D$ whose lift $\tf$ satisfies \eqref{eq:AW-condition} for a $\rho<1$ and whose dilatation $\omega$ satisfies
\[
\sup_{\z\in \D}\sqrt{|\omega(\z)|} < \frac{1-\sqrt{\rho}}{1+\sqrt{\rho}}, \quad \z\in \mathbb{D}.
\]
Then $f$ is injective and has a quasiconformal extension to $\mathbb{C}$ given by
\begin{equation*} \label{eq:extension-planar}
F(\z)=
\begin{cases}
&\hspace{-.15in}f(\z), \; \z\in\overline{\D} \\
&\hspace{-.15in}f(\z^*)+\displaystyle{\frac{(1-|\z^*|^2)h'(\z^*)}{\bar{\z^*}-(1-|\z^*|^2)\partial_z\sigma(\z^*)}+\frac{(1-|\z^*|^2)\overline{g'(\z^*)}}{\z^*-(1-|\z^*|^2)\partial_{\bar{z}}\sigma(\z^*)}}\, , \; \z^*=\frac{1}{\bar{\z}},\;,\z \notin\D.
\end{cases}
\end{equation*}
\end{theorem}
This is essentially the Ahlfors-Weill formula applied to $h$ and $\bar{g}$ separately, and becomes the classical formula exactly when $f$ is analytic. The condition on $\omega$ makes certain that the reflected surface $\Sigma^*$ is locally a graph. This is explained in Section \ref{section:planar-extension}.

We are grateful to many people for their comments. Especially, an earlier version of this paper concentrated only on the reflection $\RR$ and a two-dimensional extension, and we were encouraged to develop the techniques presented here that give the extension to space.   Finally, we were colleagues and friends of Fred Gehring, and we respectfully dedicate this paper to his memory.

\section{Circle Bundles, Convexity, and Critical Points} \label{section:convexity}

This section introduces the central notions through which the extension $\E\tf$ of $\tf$ to space is defined: bundles of Euclidean circles orthogonal to $\D$ and to $\Sigma = \tf(\D)$, respectively, and their correspondence via the family of best M\"obius approximations to $\tf$.

As a general configuration, let $B$ be a smooth, open surface in $\R$, and consider a family $\mathfrak{C}(B)$ of Euclidean circles $C_p$, at most one of which is a Euclidean line, indexed by $p \in B$, having the geometric properties:
\begin{list}{}{\setlength\leftmargin{.35in}}
\item[$(i)$] $C_p$ is orthogonal to $B$ at $p$ and ${{C}}_p \cap \overline{B}= \{p\}$;
\item[$(ii)$]  if $p_1 \ne p_2$ then ${{C}}_{p_1}\cap {{C}}_{p_2} = \emptyset$;
\item[$(iii)$] $\bigcup_{p \in B} {{C}}_{p} = {\R} \setminus \partial B$.
\end{list}
We refer to $p\in C_p$ as the base point. If $B$ is unbounded then there is no line in $\mathfrak{C}(B)$, for a line would meet $\overline{B}$ at its base point and at the point at infinity contrary to $(i)$.

If such a configuration is possible for a given $B$, it is proper to refer to $\mathfrak{C}(B)$ as a circle bundle with base space $B$.  The model example is $B = \D$ with $C_\zeta$, $\z \in \D$,  the Euclidean circle  in $\R$ that is orthogonal to $\D$ and that passes through $\zeta$ and $\zeta^* = 1/\bar{\zeta}$. When $\z=0$ the circle is a line. Next, if $T$ is a M\"obius transformation of $\overline{\R}$ then $T(\D)$ supports such a circle bundle, and simply
\begin{equation} \label{eq:bundle-invariance-1}
\mathfrak{C}(T(\D)) = T(\mathfrak{C}(\D)).
\end{equation}

Under the assumptions of Theorem \ref{theorem:extension} one can push forward $\mathfrak{C}(\D)$ to get a circle bundle of the same type on $\Sigma$ by means of a \emph{family} of M\"obius transformations $p \mapsto \Mf(p,\zeta)$ of  $p\in\overline{\R}$,  approximating $\tf$ at each $\zeta \in \D$. They are defined as follows. Let $\tilde{w}=\tf(\zeta)$. We require first that
$\Mf(\cdot,\zeta)$ maps $\C$ to the tangent plane $T_{\tilde{w}}(\Sigma)$ to $\Sigma$ at $\tilde{w}$ with $\Mf(\zeta,\zeta) = \tilde{w}$. Next, for any smooth curve  $\psi(t)$ in $\D$ with $\psi(0)=\zeta$ we  further  require that the orthogonal projection of the curve $\tf(\psi(t))$ to $T_{\tilde{w}}(\Sigma)$ has second order contact at $\tilde{w}$ with the curve $\Mf(\psi(t),\zeta)$. This is possible for harmonic mappings and by a parameter count $\Mf(p,\zeta)$ is uniquely determined; see Section \ref{section:BMA-1}, where we will also  provide a formula for $\Mf(z,\zeta)$, for  $z \in \mathbb{C}$, that is amenable to our calculations.

Now let $\mathfrak{C}(\Sigma)$ be the family of circles 
\begin{equation} \label{eq:circle-image}
{C}_{\tilde{w} }= \Mf(C_\zeta, \z),\quad \z\in \D,\,  \tw= \tf(\z).
\end{equation}
 The main work of this section is then to prove:

\begin{theorem} \label{theorem:image-bundle}
If $\tf$ satisfies \eqref{eq:AW-condition} and is injective on $\partial\D$ then $\mathfrak{C}(\Sigma)$ satisfies properties $(i)$, $(ii)$ and $(iii)$ above.
\end{theorem}
Clearly if $T$ is a M\"obius transformation of $\overline{\R}$ then $T(\Sigma)$ also supports such a circle bundle and as in \eqref{eq:bundle-invariance-1}
\begin{equation} \label{eq:bundle-invariance-2}
\mathfrak{C}(T(\Sigma)) = T(\mathfrak{C}(\Sigma)).
\end{equation}

We now define the extension $\E \tf:\overline{\R} \to \overline{\R}$ of $\tf$ by
\begin{equation} \label{eq:Ef}
\E\tf(p) =
\begin{cases}
\Mf(p,\zeta),& \quad p \in C_\zeta,\\
\tf(p),& \quad p \in \partial\D,
\end{cases}
\end{equation}
deferring discussion of the correspondence of $\partial\D$ and $\partial\Sigma$ to the next section. Note that if $p\in \overline{\D}$ then $\E\tf(p) = \tf(p)$. The disjointness of the circles $C_{\tw}$ guarantees that $\E\tf$ is injective (when $\partial \Sigma$ is a Jordan curve) and the fact that $\R = \bigcup_{\tw \in \Sigma} C_{\tw} \cup \partial\Sigma$ guarantees that it is surjective. It is obviously a homeomorphism, even real analytic, off $\partial\D$.

\bigskip

\subsection{Hyperbolic Convexity, the Auxiliary Function $\U\tf$, and Ahlfors' Schwarzian} \label{section:convexity-U-S1f}

 The proof of Theorem \ref{theorem:image-bundle} is analytic and relies on the hyperbolic convexity and a study of the critical points of the following function. For a conformal mapping $\Phi: \D \to \R$ with conformal metric $e^{2\tau(z)}|dz|^2$ on $\D$ define
 \begin{equation} \label{eq:U}
\U\Phi(z) = \frac{1}{\sqrt{(1-|z|^2)e^{\tau(z)}}}, \quad z \in \D.
\end{equation}
From Schwarz's lemma, if $M$ is any M\"obius transformation of $\D$ onto itself then
\begin{equation} \label{eq:U-invariant}
\U(\Phi \circ M) = (\U \Phi) \circ M.
\end{equation}
We will be considering the critical points of $\U\Phi$, for various $\Phi$, and the importance of \eqref{eq:U-invariant} is that we can shift a critical point to be located at the origin, and not introduce any additional critical points.

\begin{lemma} \label{lemma:U-convex}
If $\tf$ satisfies \eqref{eq:AW-condition} and $T$ is any M\"obius  transformation of $\overline{\R}$ then $\U(T\circ \tf)$ is hyperbolically convex. If $\tf$ is injective on $\partial\D$ then $\U(T \circ \tf)$ has at most one critical point in $\D$.
\end{lemma}

To explain the terms, a real-valued function $u$ on $\D$ is {hyperbolically convex} if
\begin{equation} \label{eq:u''>0}
(u\circ \gamma)''(s) \ge 0
\end{equation}
for all hyperbolic geodesics $\gamma(s)$ in $\D$, where $s$ is the hyperbolic arclength parameter. A special case of Theorem 4 in \cite{cdo:injective-lift} tells us that when $\tf$ satisfies  \eqref{eq:AW-condition} the function
$\U \tf$ is hyperbolically convex. The principle is that an upper bound for the Schwarzian leads to a lower bound for the Hessian of $\U \tf$, and this leads to \eqref{eq:u''>0} when $\U\tf$ is restricted to a geodesic. It is important more generally, as in the present lemma,  that convexity holds for $\U(T \circ \tf)$ where $T$ is a M\"obius transformation of $\overline{\R}$. The mapping $T\circ \tf$ is typically not harmonic, but it is still conformal  and in  \eqref{eq:U} the function $e^\tau$ for $\U(T \circ \tf)$ is
 \begin{equation} \label{eq:tau}
 e^\tau = (|T'| \circ \tf)e^\sigma.
 \end{equation}

To motivate the functions $\U\tf$, and $\U(T\circ \tf)$, let
\[
\lambda_\D(z)^2|dz|^2= \frac{1}{(1-|z|^2)^2}|dz|^2
\]
be the Poincar\'e metric for $\D$ and, supposing that $\tf$ is injective,  let  $\lambda_\Sigma^2 |d\tw|^2$ be the conformal metric on $\Sigma$ with
\[
\tf^*(\lambda_\Sigma^2|d\tw|^2) = \lambda_\D^2 |dz|,
\]
so that $\tf$ is an isometry. 
Since $\tf^*(|d\tw|^2) = e^{2\sg}|dz|^2$ we have
\begin{equation} \label{eq:lambda-Sigma}
(\lambda_\Sigma \circ \tf)(z)  = \frac{1}{(1-|z|^2)e^{\sg(z)}} \quad \text{or} \quad \lambda_\Sigma \circ \tf = e^{- \sigma}\lambda_\D,
\end{equation}
and
\begin{equation} \label{eq:Poincare-Sigma}
\U{\tf} = (\lambda_\Sigma\circ \tf)^{1/2}.
\end{equation}
If $f$ is analytic and injective in $\D$ and the plane domain $\Omega =f(\D)$ replaces the minimal surface $\Sigma$, then $\lambda_\Sigma = \lambda_\Omega$ is the Poincar\'e metric for  $\Omega$. It is reasonable to consider $\lambda_\Sigma$ as the Poincar\'e metric of $\Sigma$ in the case of minimal surfaces. 
In \cite{cop:john-nehari} it was shown that the hyperbolic convexity of $\lambda_{T(\Omega)}^{1/2}$ for any M\"obius transformation $T$ is a necessary and sufficient condition for a function to  satisfy the Nehari injectivity condition \eqref{eq:nehari-2}. The first part of Lemma \ref{lemma:U-convex} is the analogous result for harmonic maps of the sufficient condition.

\bigskip

The proof of Lemma \ref{lemma:U-convex}  employs a version of the Schwarzian for curves introduced by Ahlfors in \cite{ahlfors:schwarzian-R^n}.  Let $\varphi : (a,b)\to  \R$ be of class $C^3$ with
$\varphi'(x)\neq0$. As a generalization of the real part of the analytic Schwarzian, Ahlfors defined
\begin{equation}
\label{eq:S1}
S_1\varphi = \frac{\langle \varphi''',\varphi'\rangle}{\|\varphi'\|^2}
- 3\frac{\langle \varphi'',\varphi'\rangle^2}{\|\varphi'\|^4}
+ \frac32\frac{\|\varphi''\|^2}{\|\varphi'\|^2}. 
\end{equation}
If $T$ is a M\"obius transformation of $\overline{\R}$ then
\begin{equation} \label{eq:S1-invariant}
S_1(T\circ \varphi) = S_1\varphi,
\end{equation}
a crucial invariance property.

Ahlfors' interest was in the relation of $S_1\varphi$ to the change in  cross ratio under $\varphi$, while another geometric property of $S_1\varphi$ was discovered by Chuaqui and Gevirtz in  \cite{chuaqui-gevirtz:S1}. Namely, if
 \[
 v=\|\varphi'\|
 \]
 then
 \begin{equation} \label{eq:S1-curvature}
 S_1\varphi = \left(\frac{v'}{v}\right)'-\frac{1}{2}\left(\frac{v'}{v}\right)^2 + \frac{1}{2}v^2\kappa^2,
 \end{equation}
 where $\kappa$ is the curvature of the curve $x\mapsto\varphi(x)$. We will need the connection  between $S_1$ and the Schwarzian for harmonic maps, namely
\[
S_1\tf(x) \le \Re\{\Sch f(x)\} + e^{2\sg(x)}|K(\tf(x))|, \quad -1 < x <1;
\]
see Lemma 1 in \cite{cdo:injective-lift}. Thus if $f$ satisfies \eqref{eq:AW-condition} then
\begin{equation} \label{eq:S1-bound}
S_1\tf(x) \le \frac{2
\rho}{(1-x^2)^2}, \quad -1 < x < 1.
\end{equation}

 We proceed with:

\begin{proof}[Proof of Lemma \ref{lemma:U-convex}]
To show $\U(T\circ \tf)$ is hyperbolically convex it suffices to show that $\U(T\circ \tf)''(s)\ge 0$ along the diameter $(-1,1)$.
For $x\in(-1,1)$  let $\varphi(x) = (T \circ \tf)(x)$.  From \eqref{eq:S1-invariant} and  \eqref{eq:S1-bound},
 \[
 S_1\varphi(x) = S_1 \tf(x) \le \frac{2\rho}{(1-x^2)^2}.
 \]
Next let
\begin{equation} \label{eq:v}
  v(x) = |\varphi'(x)|= e^{\tau(x)}.
\end{equation}
 From \eqref{eq:S1-curvature},
\begin{equation} \label{eq:v-and-S1}
\left(\frac{v'(x)}{v(x)}\right)'  -\frac{1}{2}\left(\frac{v'(x)}{v(x)}\right)^2 \le S_1\varphi(x)  \le \frac{2\rho}{(1-x^2)^2}. %
\end{equation}
Let $2P$ denote the left-hand side, so that
\begin{equation} \label{eq:p<=}
2P(x) \le \frac{2\rho}{(1-x^2)^2}, \quad -1 < x <1.
\end{equation}
The function
\[
V=v^{-1/2}
\]
 satisfies the differential equation
\begin{equation} \label{eq:V''}
V''+PV = 0
\end{equation}
and the function
\begin{equation} \label{eq:W}
W(x) = \frac{V(x)}{\sqrt{1-x^2}}
\end{equation}
is precisely $\U(T\circ \tf)$ restricted to $-1 <x <1$. If we give $-1 <x <1$ its hyperbolic parametrization,
\[
s=\frac{1}{2}\log\frac{1+x}{1-x}, \quad x(s) = \frac{e^{2s}-1}{e^{2s}+1}, \quad x'(s) = 1-x(s)^2,
\]
 a calculation produces
\[
\frac{d^2}{ds^2}W = \left(\frac{1}{(1-x^2)^2}-P(x)\right)(1-x^2)^2W(x), \quad x=x(s),
\]
and this is nonnegative by  \eqref{eq:p<=}.

For the second part of the lemma, suppose that $\U(T\circ\tf)$ has two critical points. Composing $\tf$ with a M\"obius transformation of $\D$ onto itself we may locate these at $0$ and $a$, $0<a<1$. By convexity they must give absolute minima of $\U(T\circ\tf)$ in $\D$, and the same must be true of $\U(T\circ\tf)(x)$ for $0\le x \le a$. Hence $\U(T\circ \tf)$ is constant on $[0,a]$ and thus constant on $(-1,1)$ because it is real analytic there.

It follows that the function $v(x) = e^{\tau(x)}$ is a constant multiple of $1/(1-x^2)^2$. But then $V(x)= v(x)^{-1/2}$ is constant multiple of $\sqrt{1-x^2}$, and from the differential equation \eqref{eq:V''} we conclude that $p(x) = 1/(1-x^2)^2$. In turn, from \eqref{eq:S1-curvature} and  \eqref{eq:v-and-S1} this forces the curvature $\kappa$ of the curve $x \mapsto (T\circ \tf)(x)$ to vanish identically. Thus $T\circ \tf$ maps the interval $(-1,1)$ onto a line with speed $|\varphi'(x)|=v(x) = 1/(1-x^2)$, and so $\varphi(1) = \varphi(-1) = \infty$. This  violates the assumption that $\tf$, hence $T\circ \tf$, is injective on $\partial\D$.
 \end{proof}

\subsection{Critical Points of $\U\tf$} \label{section:critical-points}

The crucial connection between critical points of $\U(T \circ \tf)$ and the circles in ${\mathfrak{C}}(\Sigma)$ is that inversion can be used to produce a critical point exactly when the center of inversion is on the circle. In fact, this is an analytical characterization of the circles in $\mathfrak{C}(\D)$ and in $\mathfrak{C}(\Sigma)$.

 We denote M\"obius inversion by
\begin{equation} \label{eq:I_q}
I_q(p) = \frac{p-q}{\|p-q\|^2}, \quad p \in \R, \quad \text{with} \quad \|DI_q(p)\| = \frac{1}{\|p-q\|^2}.
\end{equation}

\begin{lemma} \label{lemma:inversion-critical-point}
A point ${q} \in \R$ lies on the circle ${C}_{\tilde{w}} \in \mathfrak{C}(\Sigma)$, $\tilde{w}=\tf(\z)$,  if and only if  $\U(I_{{q}} \circ \tf)$ has a critical point at $\z$.
\end{lemma}

\begin{proof}
The statement also applies to the bundle $\mathfrak{C}(\D)$ by taking $\tf$ to be the identity. Consider this case first, and assume further that $\z=0$. The lemma then says that $q \in C_0$, the vertical line in $\R$ through the origin, if and only if
\[
\U I_q(z) = \frac{\|z-q\|}{\sqrt{1-|z|^2}}
\]
has a critical point at $0$. This is easy to check by direct calculation. The result for $\z \in \D$ follows from this and from \eqref{eq:U-invariant} letting $M$ be a M\"obius transformation of the disk mapping $0$ to $\z$, since $M'$ does not vanish in $\D$ and the extension of $M$ to space maps $C_0$ to $C_\z$.

Now take a general $\tf$, fix $\zeta \in \D$, and let $\tilde{w} = \tf(\z)$. Observe that  $\U(I_{\tilde{q}}\circ \tf)$ has a critical point at $\z$ if and only if $\U(I_{\tilde{q}} \circ \Mf(\cdot,\z))$ has as well, because $\Mf(\cdot,\zeta)$ and $\tf$ agree at $\zeta$ to first order.  Suppose $\tilde{q} \in \tilde{C}_{\tilde{w}}$ with $\tilde{q}=\Mf(q,\z)$, $q \in C_\z$. As a M\"obius transformation of $\R$ the mapping $I_{\tilde{q}} \circ \Mf(\cdot,\z)$ sends $q$ to $\infty$, and so up to an affine transformation it \emph{is} $I_q$. But from the first part of the lemma $\U I_q$ has a critical point at $\zeta$. This proves necessity, and the argument can be reversed to prove sufficiency.
\end{proof}

Knowing how to produce a critical point, we now show what happens when there \emph{is} one (and only one).
\begin{lemma} \label{lemma:bounded-equivalences}
Let $\tf$ satisfy \eqref{eq:AW-condition} and be injective on $\partial\D$. Let $T$ be a M\"obius transformation of $\overline{\mathbb{R}^3}$. The following are equivalent:
\begin{list}{}{\setlength\leftmargin{.25in}}
\item[$(i)$]  $\U(T\circ\tf)$ has a critical point.
\item[$(ii)$]  $(T\circ \tf)(\D)$ is bounded.
\item[$(iii)$] $\U(T\circ \tf)(re^{i\theta})$ is eventually increasing along each radius $[0,e^{i\theta})$.
\item[$(iv)$]  $\U(T\circ\tf)(z) \to \infty$ as $|z|\to 1$.
\end{list}
\end{lemma}

\begin{proof}
If $(iv)$ holds there is an interior minimum so $(iv) \implies (i)$ is immediate.

Suppose $(i)$ holds. We follow the notation in Lemma \ref{lemma:U-convex}. We may assume the critical point is at the origin. The value $\U(T \circ \tf)(0)$ is the absolute minimum for  $\U(T \circ \tf)$ in $\D$ and so
\[
e^{\tau(z)} \le \frac{e^{\tau(0)}}{1-|z|^2}, \quad z \in \D.
\]
Thus  $\tau$ remains finite in $\D$ and $\infty$ cannot be a point on $(T\circ \tf)(\D)$.

To show that $(T\circ \tf)(\D)$ is bounded we first work along $[0,1)$. The hyperbolically convex function $W(x) =\U(T \circ \tf)(x)$ in \eqref{eq:W} cannot be constant because $0$ is the unique critical point. Hence if $x(s)$ is the  hyperbolic arclength parametrization of $[0,1)$ with $x(0)=0$
\[
\frac{d}{ds} W(x(s)) \ge a, \quad W(x(s)) \ge a s + b,
\]
for some $a, b >0$ and all $s \ge s_0 >0$. From this
\[
\begin{aligned}
v(x) = \frac{1}{V(x)^2} & \le \frac{1}{(1-x^2)\left(\frac{a}{2}\log\frac{1+x}{1-x}+b\right)^2}\\
&= -\frac{1}{a}\frac{d}{dx}\left(\frac{1}{\frac{a}{2}\log\frac{1+x}{1-x}+b}\right).
\end{aligned}
\]
Therefore
\[
\int_0^1 e^{\tau(x)}\,dx = \int_0^1v(x)\,dx <\infty,
\]
with a bound depending only on $a$, $b$, and $s_0$, so $(T\circ \tf)(1)$ is finite.

This argument can be applied on every radius $[0,e^{i\theta})$, and by compactness the corresponding numbers $a_\theta, b_\theta, s_\theta$ can be chosen positive independent of $\theta$. This proves that $T\circ\tf$ is bounded, and hence that $(i) \implies (ii)$.

For $(ii)$ $\implies$ ${(iii)}$ we can first rotate and assume $e^{i\theta} = 1$. In the notation above, we need to show for some $x_0 >0$ that $W(x)$ is increasing for $x_0\le x <1$. 

We have to follow $T$ by an inversion, so to simplify the notation let $\tf_1= T\circ \tf$ and
$\U{\tf_1}(z) = ((1-|z|^2)e^{\tau(z)})^{-1/2}$. For a $q\in \R$ to be determined let
\[
\tilde{f_2}=I_q\circ\tilde{f_1} ,\quad
\U{\tilde{f_2}}(z) = \frac{1}{\sqrt{(1-|z|^2)e^{\nu(z)}}}, \quad \nu(z) = \tau(z) - \log|\tilde{f_1}(z)-q|^2.
\]
Let $W_2(x) = \U{\tilde{f_2}}(x)$,  $x \in (-1,1)$; again we know that $W_2(x(s))$ is convex, where $s$ is the hyperbolic arclength parameter. Now
\begin{equation} \label{eq:grad-equation}
\nabla\nu(0) = \nabla\tau(0) + \frac{2}{|q|^2}(\langle {\partial_x \tf_1}(0), q\rangle, \langle {\partial_y \tilde{f_1}}(0),q\rangle).
\end{equation}
But also 
\[
\nabla \U{\tf_2}(0) = -\frac{1}{2}e^{-\nu(0)/2}\,\nabla \nu(0),
\]
and from this equation and \eqref{eq:grad-equation} it is clear we can choose $q$ to make
\[
W_2'(0) = a>0.
\]
Convexity then ensures $W_2(x(s)) \ge a s$.

To work back to $W$, write
\begin{equation} \label{eq:f1-and-f2}
\tilde{f_1} = \frac{\tilde{f_2}}{|\tilde{f_2}|^2}+q,
\end{equation}
whence
\[
\|D\tilde{f_1}\| =\frac{\|D\tilde{f_2}\|}{\|\tilde{f_2}\|^2},
\]
and
\[
W=W_2\|\tf_2\|.
\]
The assumption we make in $(ii)$ is that $\tilde{f_1}(\D) = (T\circ \tf)(\D)$ is bounded, and \eqref{eq:f1-and-f2} thus implies that  $\|\tf_2\|\ge \delta>0$. Therefore $W(x(s)) \ge a\delta s$. By convexity, there is an $x_0>0$ so that $W(x)$ is increasing for $x_0 \le x <1$.  This completes the proof that $(ii) \implies (iii)$.

Finally, if $(iii)$ holds then for each $\theta$ there exists $0 < r_\theta <1$ such that
\[
\frac{\partial}{\partial r} \U(T\circ \tf)(r_\theta e^{i\theta}) \ge a_\theta >0.
\]
 By compactness the $r_\theta$ can be chosen bounded away from $1$  and the $a_\theta$ bounded away from $0$. By hyperbolic convexity, along the tail of each radius $\U({T \circ\tf})(r(s)e^{i\theta})$ is uniformly bounded below by a linear function of the hyperbolic arclength parameter $s$, which tends to $\infty$ as $r= r(s) \rightarrow 1$.
\end{proof}

We now have:

\begin{proof}[Proof of Theorem \ref{theorem:image-bundle}]
For $(i)$, orthogonality is obvious, and suppose ${{C}}_{\tilde{w}}$ meets $\overline{\Sigma}$ at a second point $\tilde{w}'$. Then the inversion $I_{\tilde{w}'}$, which takes $\tilde{w}'$ to $\infty$, produces a critical point for $\U(I_{\tilde{w}'} \circ \tf)$ at $\zeta = \tf^{-1}(\tilde{w})$ by Lemma \ref{lemma:inversion-critical-point}. But by Lemma \ref{lemma:bounded-equivalences}, $I_{\tilde{w}'}(\Sigma)$ is bounded.

For $(ii)$, if there is a point $\tilde{w}_3 \in {{C}}_{\tilde{w}_1}\cap \tilde{{C}}_{\tilde{w}_2}$ then the inversion $I_{\tilde{w}_3}$ produces critical points for $\U(I_{\tilde{w}_3} \circ \tf)$ at the two distinct points $\z_1=\tf^{-1}(\tilde{w}_1)$ and $\z_2=\tf^{-1}(\tilde{w}_2)$, contradicting Lemma \ref{lemma:U-convex}.

Finally for $(iii)$, by definition the base points $\tilde{w}\in \Sigma$ for the circles $C_{\tilde{w}}$ cover $\Sigma$. Suppose $q \not\in \overline{\Sigma}$. Then under inversion $I_q(\Sigma)$ is bounded. Therefore by Lemma \ref{lemma:bounded-equivalences}, $I_q \circ \tf$ has a critical point, and by Lemma \ref{lemma:inversion-critical-point} the point $q$ is on some circle ${{C}}_{\tilde{w}}$, $\tilde{w} \in \Sigma$.
\end{proof}

\subsection*{Remark} The differential equations argument in Lemma \ref{lemma:U-convex} is a version of what we have called  `relative convexity' in other work,  \cite{cdo:valence}, \cite{cdo:injective-lift}. See also the paper of Aharonov and Elias \cite{aharonov:Sturm}. The relation between critical points of the Poincar\'e metric and the Ahlfors-Weill extension was the subject of \cite{co:aw}.



\section{Best M\"obius Approximations, I: Reflection across $\partial\Sigma$ and the Ahlfors-Weill Extension} \label{section:BMA-1}

To study the extension $\E\tf$ we need an expression for the best M\"obius approximations. 
The first  condition on $\Mf$ is that $p\mapsto \Mf(p,\z)$ maps $\C$ to the tangent plane $T_{\tilde{w}}(\Sigma)$,  where $\tilde{w}=\tf(\z)=\Mf(\z,\z)$. Let $\mathbf{N}$ be a unit normal vector field along $\Sigma$. At each point $\tw \in \Sigma$ we write $\Hyp_{\tw}(\Sigma)$ for  the hyperbolic (upper) half-space over $T_{\tw}(\Sigma)$ determined by $\mathbf{N}_{\tw}$. Then $p\mapsto \Mf(p,\z)$ is an isometry of $\Hyp$ with  $\Hyp_{\tw}(\Sigma)$, but the  fact that these half-spaces vary along $\Sigma$, unlike when $\Sigma$ is planar, is at the root of the complications in our analysis.

 In appropriate coordinates on the range we can take $T_{\tilde{w}}(\Sigma)=\mathbb{C}$ and
 regard $\Mf(z,\zeta)$, $z\in \mathbb{C}$, as an ordinary complex M\"obius transformation of $\overline{\C}$. 
 With this convention, $z\mapsto \Mf(z,\zeta)$, $z=x+iy$,  depends on six real parameters, each depending on $\zeta$, and once these are determined so is $\Mf(p,\zeta)$ for $p\in\overline{\R}$.  Specifying $\Mf(\zeta,\zeta) = \tilde{w}$ fixes two of the parameters. Next, let $\psi(t)$ be a smooth curve in $\D$ with $\psi(0)=\z$.  To match $\tf$ and $\Mf$ along $\psi$ to first order at $\tilde{w}$  it suffices to have $ \partial_x \tf(\zeta) = \partial_x \Mf(\z,\z) $    because, using that  $\tf$ and $\Mf$ are conformal, the same will then be true of the $y$-derivatives. It takes two more real parameters in $\Mf$ to ensure this. We can use the final two parameters to make the orthogonal projection of $\partial_{xx}\tf(\z)$ onto $T_{\tilde{w}}(\Sigma)$ equal to  $\partial_{xx}\Mf(\z,\z)$, and because $\tf$ and $\Mf$ are harmonic (as functions of $x$ and $y$) the second $y$-derivatives also agree. Finally, a calculation using again the conformality of $\tf$ and of $\Mf$ shows that $\partial_{xy}\Mf(\zeta,\zeta)$ agrees with the tangential component of $\partial_{xy}\tf(\zeta)$.

Requiring equality of the various derivatives of $\tf$ and $\Mf$ is an alternate way of defining $\Mf$ and can be put to use to develop a formula for $\Mf(z,\z)$ for $z=x+iy \in \mathbb{C}$. We have found that the most convenient expression is
\begin{equation} \label{eq:M-formula}
\Mf(z,\z)= \tf(\z)+\Re\{m(z,\z)\}\partial_\xi\tf(\z)+ \Im\{m(z,\z)\}\partial_\eta\tf(\z), \quad \zeta = \xi + i \eta,
\end{equation}
where
\begin{equation} \label{eq:BMA-m}
m(z,\z) = \frac{z-\z}{1-\partial_\z\sigma(\z)(z-\z)}.
\end{equation}
Note the two special values
\begin{equation} \label{eq:m-values}
m(\z,\z)=0, \quad m(\z^*,\z) = \frac{1-|\z|^2}{\bar{\z}-\partial_\z\sigma(\z)(1-|\z|^2)}. 
\end{equation}
We verify that \eqref{eq:M-formula} meets the requirements in the preceding paragraph.

Immediately $\Mf(\z,\z) = \tf(\z)$, from $m(\z,\z)=0$. Next, with $\z$ fixed and $z=x+iy$ varying, differentiate the right-hand side of \eqref{eq:M-formula} with respect to $x$ and set $z = \zeta$. As
\begin{equation} \label{eq:partial_x-m}
\partial_xm(z,\z) = \frac{1}{(1-\partial_\z\sigma(\z)(z-\z))^2},
\end{equation}
we have simply $\partial_xm(z,\z)|_{z=\z} =1$, thus
\[
\partial_x\Mf(z,\z)|_{z=\z} = 1\cdot \partial_\xi\tf(\z).
\]
Taking the second $x$-derivative we first have,
\[
\partial_{xx}m(z,\z) = \frac{2\partial_\z\sigma(\z)}{(1- \partial_\z\sigma(\z)(z-\z))^3},
\]
whence
\[
\partial_{xx}m(z,\z)|_{z=\z} = 2\partial_\z\sigma(\z),
\]
and
\[
\begin{aligned}
\partial_{xx}\Mf(z,\z)|_{z=\z} & = 2\Re\{\partial_\z\sigma(\z)\}\partial_\xi\tf(\z) + 2\Im\{\partial_\z\sigma(\z)\}\partial_\eta\tf(\z)\\
&= \partial_\xi\sigma(\z)\partial_\xi\tf(\z) -\partial_\eta\sigma(\z)\partial_\eta\tf(\z).
\end{aligned}
\]
Next, projecting onto $\partial_\xi\tf(\z)$ gives
\[
\langle\partial_{xx}\Mf(z,\z)|_{z=\z}, \partial_\xi\tf(\z) \rangle = \partial_\xi\sigma(\z)\langle \partial_\xi\tf(\z), \partial_\xi\tf(\z) \rangle = e^{2\sigma(\z)}\partial_\xi\sigma(\z).
\]
On the other hand, projecting $\partial_{\xi\xi}\tf(\z)$ onto $\partial_\xi\tf(\z)$ we get
\[
\langle \partial_{\xi\xi}\tf(\z), \partial_\xi\tf(\z) \rangle = \left.\frac{1}{2} \partial_\xi\langle \partial_\xi\tf, \partial_x\tf)\rangle\right|_{\z} = \frac{1}{2} \partial_\xi(e^{2\sigma})(\z) = e^{2\sigma(\z)}\partial_\xi\sigma(\z),
\]
as we should. Similarly, using
\[
\langle \partial_{\xi\xi}\tf(\z),\partial_\eta\tf(\z)\rangle = -\langle \partial_{\eta\eta}\tf(\z),\partial_\eta\tf(\z)\rangle = \left.-\frac{1}{2}\partial_\eta\langle \partial_\eta\tf,\partial_\eta\tf\rangle\right|_{\z}=-e^{2\sigma(\z)}\partial_\eta\sg(\z)
\]
we again get the correct equality. This completes the verification of \eqref{eq:M-formula}.

As a point of reference we want to see the form that \eqref{eq:M-formula} takes when $\tf$ is analytic.  In that case
\[
\sigma(\z) = \log|\tf'(\z)|, \quad \text{and} \quad \partial_\z \sigma(\z) = \frac{1}{2}\frac{\tf''(\z)}{\tf'(\z)},
\]
and
\begin{equation} \label{eq:m-analytic}
m(z,\z)= \frac{z-\z}{1-\frac{1}{2}(z-\z)\displaystyle{\frac{\tf''(\z)}{\tf'(\z)}}}.
\end{equation}
Then using the Cauchy-Riemann equations,
\begin{equation}  \label{eq:Mf-analytic}
\Mf(z,\z) = \tf(\z) + m(z,\z)\tf'(\z).
\end{equation}
For derivatives of $\Mf(z,\z)$ with respect to $z$ we obtain:
\[
 \partial_z\Mf(z,\z) = \frac{\tf'(\z)}{\left(1-\frac{1}{2}(z-\z)\displaystyle{\frac{\tf''(\z)}{\tf'(\z)}}\right)^2},
\]
\[
\partial_{zz}\Mf(z,\z) = \frac{\tf''(\z)}{\left(1-\frac{1}{2}(z-\z)\displaystyle{\frac{\tf''(\z)}{\tf'(\z)}}\right)^3},
\]
showing  second order contact between $\tf$ and $\Mf$ when $z=\z$.

\subsection{Reflection Across $\partial\Sigma$} \label{subsection:reflection}

The existence of the bundle $\mathfrak{C}(\Sigma)$ allows us to define a reflection of $\Sigma$ across its boundary. If $\tw \in \Sigma$ the circle $C_{\tw}$ intersects the tangent plane $T_{\tw}(\Sigma)$ orthogonally at a diametrically opposite point $\tw^*$ outside $\overline{\Sigma}$, and  we write
\begin{equation} \label{eq:reflection-1}
\tw^* = \RR(\tw),
\end{equation}
for this correspondence.  Equivalently, if $\tw = \tf(\z) =\Mf(\z,\z)$ then
\begin{equation} \label{eq:reflection-BMA}
\tw^* = \Mf(\z^*,\z)=\E\tf(\z^*), \quad \z^*=1/\bar{\z}.
\end{equation}
In this section we will show that $\RR$ fixes $\partial\Sigma$ pointwise.

From $\tw^*=\Mf(\z^*,\z)$ and $m(\z,\z)=0$ we obtain
\begin{equation} \label{eq:w^*-w}
\tw^* -\tw
= \Re\{m(\z^*,\z)\}\partial_x\tf(\z) + \Im\{m(\z^*,\z)\}\partial_y\tf(\z) .
\end{equation}
Moreover, from \eqref{eq:U} we have
\begin{equation} \label{eq:U_z}
\partial_z\log\U\tf(\zeta) = \frac{\bar{\z} - \partial_z\sigma(\z)(1-|\z|^2)}{2(1-|\z|^2)} = \frac{1}{2m(\z^*,\z)},
\end{equation}
and so we also obtain
\begin{equation} \label{eq:C-diameter}
\|\tw^*-\tw\| = \frac{e^{\sg(\zeta)}}{\|\nabla \log \U\tf(\zeta)\|}.
\end{equation}
This is the length of the diameter of ${C}_{\tw}$ and we want to see that it tends to $0$ as $\tw$ approaches $\partial\Sigma$. We formulate the result as:  
\begin{theorem} \label{theorem:reflection}
Let $d$ denote the spherical metric on $\overline{\R}$. If $\tf$ satisfies \eqref{eq:AW-condition} and is injective on $\partial\D$ then
\[
d(\RR(\tf(\z)),\tf(\z)) \rightarrow 0 \quad \text{as} \quad |\z| \rightarrow 1.
\]
\end{theorem}

\begin{proof}
 We divide the proof into the cases when $\U\tf$ has one critical point and when it has none. We work in the spherical metric because, first, $\tf$ has a spherically continuous extension to $\partial\D$ (by \cite{cdo:injective-lift}), and second, when $\U\tf$ has no critical points we have to allow for shifting $\tf$ by a M\"obius transformation.

Suppose $\U\tf$ has a unique critical point, which, by \eqref{eq:U-invariant}, we can take to be at $0$. The proof of Lemma~\ref{lemma:bounded-equivalences} shows that there is an $a >0$ such that along any radius $[0, e^{i\theta})$
\[
(1-r^2)\frac{\partial}{\partial r} \U\tf(re^{i\theta}) \ge a
\]
for all $r \ge r_0>0$. (This corresponds to $dW/ds \ge a$ in the proof of Lemma~\ref{lemma:bounded-equivalences}, where $s$ is the hyperbolic arclength parameter.) It follows that
\begin{equation} \label{eq:grad-u-lower-bound}
(1-|\z|^2)\|\nabla \U\tf(\z)\| \ge a >0,
\end{equation}
for all $|\z| \ge r_0 >0$.

From \eqref{eq:C-diameter}
\[
\begin{aligned}
|\RR(\tf(\z)) - \tf(\z)| &= \frac{e^{\sg(\z)}}{\| \nabla \log \U\tf(\z)\|}
=\frac{\U\tf(\z)e^{\sg(\z)}}{\|\nabla \U\tf(\z)\|}\\
&= \frac{1}{\U\tf(\z)}\frac{1}{(1-|\z|^2)\|\nabla \U\tf(\z)\|}.
\end{aligned}
\]
This tends to $0$ as $|\z|\rightarrow 1$ because $\U\tf$ becomes infinite (Lemma \ref{lemma:bounded-equivalences}) and $(1-|\z|^2)\|\nabla \U\tf(\z)\|$ stays bounded below. 

Next, supposing that $\U\tf$ has no critical point, we produce one. That is, let $T$ be a M\"obius  transformation so that $\U{(T \circ \tf)}$ has a critical point at $0$. The preceding argument can be repeated verbatim to conclude that
\begin{equation} \label{eq:diameter-shrinks}
\|\RR'(T(\tf(\z))) - T(\tf(\z))\|\rightarrow 0 \quad \text{as} \quad  |\z| \rightarrow 1,
\end{equation}
where $\RR'$ is the reflection for the surface $\Sigma'=T(\Sigma)$. If the reflections were conformally natural, if we knew that $\RR' \circ T= T \circ \RR$, then  we would be done.  Instead, we argue as follows.

Let $\z\in \D$, $\z \ne 0$. The number $\|\RR'(T(\tf(\z))) - T(\tf(\z))\|$ is the length of the diameter of the circle $C_{T(\tf(\z))}$ based at $T(\tf(\z))$  that defines the reflection $\RR'$, and it tends to $0$ by \eqref{eq:diameter-shrinks}. But now, if $C_{\tf(\z)}$ is the circle based at $\tf(\z)$, for the surface $\Sigma$  then  $\RR(\tf(\z))$ is  on  $C_{\tf(\z)}$ (diametrically opposite to $\tf(\z)$), and then $T(\RR(\tf(\z))) \in C_{T(\tf(\z))}$. Therefore $\|T(\RR({\tf}(\z))) - T(\tf(\z))\|\rightarrow 0$ as $|\z|\rightarrow 1$, whence in the spherical metric $d(\RR({\tf}(\z)), \tf(\z))$ 
tends to $0$ as well and the proof is complete.
\end{proof}

\subsection*{Remark} Theorem \ref{theorem:reflection} shows that $\RR$ is indeed a reflection across $\partial\Sigma$, and that the extension $\E\tf$ is continuous at $\partial\D$. In Section \ref{subsection:reflection-qc} we will show that $\RR$ is quasiconformal, a property needed for the proof of Theorem \ref{theorem:planar-extension} on the quasiconformal extension of the \emph{planar} harmonic mapping $f=h+\bar{g}$. This separate fact is not necessary for the proof of Theorem \ref{theorem:extension}, but it follows from limiting cases of the estimates in Section \ref{subsection:dilatation}.

 Let $\Sigma^*=\RR(\Sigma)$. Then $\Sigma \cup \partial\Sigma \cup \Sigma^*$ is a topological sphere that is the image of $\C$ by the quasiconformal mapping $\E\tf$ of $\overline{\R}$, in other words it is a quasisphere.  Or,  one might also regard $\Sigma$ as a (nonplanar) quasidisk, and given the many analytic and geometric characterizations of planar quasidisks (see for example Gehring's survey \cite{gehring:quasidisks}) one might ask if any  have analogues for $\Sigma$. Some results in this direction are due to W. Sierra, who  has shown in \cite{sierra:harmonic-john} that if $\tf$ satisfies \eqref{eq:AW-condition} then $\Sigma$ is a John domain (a John surface) in its metric geometry, and if $\tf$ is also bounded then $\Sigma$ is linearly connected. Both these notions come from the geometry of planar quasidisks and we will not define therm here, see \cite{cop:john-nehari}.   Although $\Sigma^*$ will most likely not be a minimal surface, one might expect it also to have these properties.

\subsection{The Ahlfors-Weill Extension} It is possible to express the reflection $\RR$ in terms intrinsic to the surface $\Sigma$.
Recall the function $\lambda_\Sigma$ from \eqref{eq:Poincare-Sigma}, with $\lambda_\Sigma \circ \tf = (\U\tf)^2$.
Using \eqref{eq:w^*-w} and \eqref{eq:U_z} it is easy to verify that
\begin{equation} \label{eq:reflection-intrinsic}
\RR(\tw) = \tw + 2J(\nabla \log \lambda_{\Sigma}(\tw)),
\end{equation}
where,  following Ahlfors, $J$ is the M\"obius inversion centered at the origin,
\begin{equation} \label{eq:J}
J(p) = \frac{p}{|p|^2}.
\end{equation}

The formula \eqref{eq:reflection-intrinsic} will be important in Section \ref{section:planar-extension}. Here, we make contact with the classical Ahlfors-Weill extension, which, when $\tf$ is analytic in $\D$, can be written as
\[
F(z) =
\begin{cases}
\tf(z), \quad z \in \overline{\D},\\
\tf(\zeta) + \displaystyle{\frac{(1-|\zeta|^2)\tf'(\z)}{\bar{\zeta} - \frac{1}{2}(1-|\zeta|^2)\displaystyle{\frac{\tf''(\zeta)}{\tf'(\zeta)}}}}
  =\Mf(1/\bar{\z},\z) \quad \zeta = 1/\bar{z}, \,z \in \mathbb{C} \setminus \overline{\D}.
\end{cases}
\]
Ahlfors and Weill did not express their extension in this form; see \cite {co:aw}.

Alternatively, if $\lambda_\Omega|dw|$ is the Poincar\'e metric on $\Omega = f(\D)$ then
\[
F(z) =
\begin{cases}
\tf(z), \quad z \in \overline{\D},\\
\tf(\zeta) +\displaystyle{\frac{1}{\partial_w\log\lambda_\Omega(\tf(\zeta))}}, \quad \zeta = 1/\bar{z}, \,z \in \mathbb{C} \setminus \overline{\D}.
\end{cases}
\]
The equation \eqref{eq:reflection-intrinsic} for the reflection gives exactly
\begin{equation} \label{eq:aw-extension-gradient}
\RR(\tf(\zeta)) = \tf(\z) + \frac{1}{\partial_w\log\lambda_\Omega(\tf(\z))}
\end{equation}
 when $\tf$ is analytic.

 The  reflection defining the Ahlfors-Weill extension was expressed in a form like \eqref{eq:aw-extension-gradient} also by Epstein \cite{epstein:reflections}. Still another interesting geometric construction, using Euclidean circles of curvature, was given by D. Minda \cite{minda:reflections}.

The Ahlfors-Weill reflection is conformally natural, meaning in this case that if $M$ is a M\"obius transformation of $\overline{\mathbb{C}}$ and $ \Omega'=M(\Omega)$ with corresponding reflection $\RR'$, then
$
\RR' \circ M = M \circ \RR$.
From the perspective of the present paper this is so because all the tangent planes $T_z(\Omega)$ to $\Omega$ can be identified with $\mathbb{C}$, which is preserved by the extensions to $\overline{\mathbb{R}^3}$ of the M\"obius transformations. In the more general setting, if $M$ is a M\"obius transformation of $\overline{\R}$ then $\Sigma'=M(\Sigma)$ also supports a circle bundle $\mathfrak{C}(\Sigma')$ and hence an associated reflection $\RR'$. But while it is true that $\mathfrak{C}(M(\Sigma))=M(\mathfrak{C}(\Sigma))$ it is not true that
\[
\RR'\circ M = M \circ \RR,
\]
i.e.,  the reflection is not conformally natural. The reason is that the reflections $\RR$ and $\RR'$ use tangent planes for $\Sigma$ and $\Sigma'$, while $M$ may map tangent planes for $\Sigma$ to tangent planes or tangent spheres for $\Sigma'$. We do not know how to define a conformally natural reflection, at least one that is suited to our analysis.


 \subsection{A Bound on $\|\nabla\log\U\tf\|$ and a Classical Distortion Theorem}

 In the proof of Theorem \ref{theorem:reflection} we needed the lower bound \eqref{eq:grad-u-lower-bound} on $\|\nabla \U\tf\|$.
  In Section \ref{subsection:dilatation}, where we bound the dilatation of the extension $\E\tf$ to prove its quasiconformality, we will need a corresponding upper bound (to be used again in connection to \eqref{eq:C-diameter}). We state the result as

\begin{lemma} \label{lemma:grad-u-upper-bound}
 If $\tf$ satisfies \eqref{eq:AW-condition}  then
   \begin{equation}  \label{eq:grad-u-upper-bound}
\|\nabla \log \U\tf(\z)\| \le \frac{\sqrt{2}}{1-|\z|^2}.
\end{equation}
 \end{lemma}

 \begin{proof}
 From \eqref{eq:U_z}, we want to show
\begin{equation} \label{eq:sigma_z-bound}
 \left|\partial_z\sigma(\z) - \frac{\bar{\z}}{1-|\z|^2}\right| \le \frac{\sqrt{2}}{1-|\z|^2}.
\end{equation}
For this we first derive a lower bound for $\nabla|\partial_z\sigma|$. Let $\tau=|\partial_z\sigma|$, so that $\tau^2=\partial_z\sigma\partial_{\bar{z}}\sigma$ and
\[
2\tau\partial_z\tau = \partial_{zz}\sigma\partial_z\sigma+\partial_{z,\bar{z}}\sigma\partial_z\sigma =(\partial_{z\bar{z}}\sigma+|\partial_z\sigma|^2)\partial_z\sigma +(\partial_{zz}\sigma-(\partial_z\sigma)^2)\partial_z\sigma.
\]
Then
\[
2\tau|\partial_z\tau| \ge (\partial_{z\bar{z}}\sigma+|\partial_z\sigma|^2)\tau-|\partial_{z\bar{z}}\sigma-(\partial_z\sigma)^2|\tau \ge \tau^3-\frac{\tau}{(1-|\z|^2)^2},
\]
because in the first term $\sigma_{z\bar{z}} \ge 0$ and in the second because of \eqref{eq:AW-condition}. Thus
\[
|\partial_z\tau| \ge \tau^2 - \frac{1}{(1-|\z|^2)^2}.
\]

The desired estimate at $\z=0$ is
\begin{equation} \label{eq:tau-at-0}
\tau(0)=|\partial_z\sigma(0)| \le \sqrt{2},
\end{equation}
 and this will follow by showing that an initial condition $a=\tau(0) > \sqrt{2}$ leads to the contradiction that $\tau$ becomes infinite in $\mathbb{D}$. To this end, consider $v(t)=\tau(\z(t))$ along arc length parametrized integral curves $t\mapsto \z(t)$ to $\nabla\tau$. There exists such an integral curve starting at the origin because
\[
\|\nabla \tau(0)\| = 2|\partial_z\tau(0)| \ge 2\tau(0)^2-2 >0.
\]
The function $v(t)$ satisfies
\[
v'(t) = 2|\partial_z\tau(z(t))| \ge v^2(t)-\frac{1}{(1-|\z(t)|^2)^2} \ge v^2(t) - \frac{1}{(1-t^2)^2},
\]
since $|\z(t)| \le t$.

We compare $v(t)$ with the solution $y(t)$ of
\[
y'=y^2 -\frac{1}{(1-t^2)^2}, \quad y(0) = a,
\]
which is given by
\[
y = \frac{1}{2}\frac{n''}{n'},
\]
where
\[
n(t)=\frac{n_0(t)}{1-an_0(t)}, \quad \text{and} \quad n_0(t) =\frac{1}{\sqrt{2}}\frac{(1+t)^{\sqrt{2}}-(1-t)^{\sqrt{2}}}{(1+t)^{\sqrt{2}} + (1-t)^{\sqrt{2}}}.
\]
Because $a>\sqrt{2}$ there exists $0 < t_0<1$ for which $an_0(t_0)=1$. The function $y(t)$ is increasing for $0 \le t <t_0$ and  becomes unbounded there. There are then two possibilities. Either $v(t)$ becomes infinite before or at $t_0$, or the integral curve ceases to exist before that time. But while $v(t)$ is finite it is bounded below by $y(t) \ge a$, hence $|\nabla\tau|$ does not vanish, as shown above, so the integral curve can be continued. We conclude that $v(t)$ must become infinite before or at $t_0$, and this contradiction shows that \eqref{eq:tau-at-0} must hold.

To deduce \eqref{eq:sigma_z-bound} at an arbitrary point $\z_0 \in \mathbb{D}$ we consider
\[
\tf_1(z) = \tf(M(z)), \quad M(z) = \frac{z+\z_0}{1+\bar{\z_0}z}.
\]
Then $\tf_1$ satisfies \eqref{eq:AW-condition} and its conformal factor is
\[
e^{\sigma_1(z)} = e^{\sigma(M(z))}|M'(z)|.
\]
From this
\[
\partial_z\sigma_1(0) = (1-|\z_0|^2)\partial_z\sigma(\z_0) -\bar{\z_0},
\]
and \eqref{eq:sigma_z-bound} at $\z_0$ is obtained from $|\partial_z\sigma_1(0)| \le \sqrt{2}$.
  \end{proof}


\subsection*{Remarks} Suppose equality holds in \eqref{eq:grad-u-upper-bound}, so in \eqref{eq:sigma_z-bound}, at some $\z_0\in \mathbb{D}$. By composing $\tf$ with a M\"obius transformation of $\mathbb{D}$ onto itself we may suppose $\z_0=0$. The argument shows that $\partial_{z\bar{z}}\sigma$ must then vanish along the integral curve $\z(t)$ from the origin. Hence the curvature of the minimal surface $\Sigma$ vanishes on a continuum and $\Sigma$ must therefore be a planar. In turn this means that $\tf = h + \alpha\tilde{h}$ for some constant $\alpha<1$ and an analytic function $h$ for which \eqref{eq:AW-condition} holds. Because $\partial_z\sigma = (1/2)(h''/h')$ we see, from the case of equality in the analytic case, that $h$ must be an affine transformation of a rotation of the function $n$.

\smallskip

Lemma  \ref{lemma:grad-u-upper-bound}, on the one hand expressed as in \eqref{eq:sigma_z-bound}, is reminiscent of the classical distortion theorem for univalent functions, see, e.g., \cite{duren:univalent}. Namely, if $f$ is analytic and injective in $\mathbb{D}$ then
\begin{equation} \label{eq:distortion-theorem}
\left| \frac{1}{2}\frac{f''(\z)}{f'(\z)} - \frac{\bar{\z}}{1-|\z|^2}\right| \le \frac{2}{1-|\z|^2},
\end{equation}
with equality holding at a point exactly when $f$ is a rotation of the Koebe function $k(\z) = \z/(1-\z)^2$. On the other hand, it was observed in  \cite{osgood:f''/f'}  that \eqref{eq:distortion-theorem} can be written in terms of the Poincar\'e metric $\lambda_\Omega|dw|$ on $\Omega = f(\mathbb{D})$ as
 \[
 \|\nabla \log \lambda_\Omega\| \le 4 \lambda_\Omega.
 \]
For the harmonic case we recall \eqref{eq:lambda-Sigma} and \eqref{eq:Poincare-Sigma} where we had $\U\tf = (\lambda_\Sigma \circ f)^{1/2}$ with $\lambda_\Sigma$ playing the role of the Poincar\'e metric on $\Sigma$. 
The bound \eqref{eq:grad-u-upper-bound} becomes
\[
\|\nabla \log\lambda_\Sigma\| \le 2\sqrt{2}\,\lambda_\Sigma.
\]

\section{Best M\"obius Approximations, II: Dependence on the base point} \label{section:BMA-2}

We return to properties of best M\"obius approximations and examine  how $\Mf(z,\z)$  varies with the base point $\z$. First, when $\tf$ is analytic and $\Mf(z,\z)$ is given by \eqref{eq:m-analytic} and
\eqref{eq:Mf-analytic} we find
\begin{equation} \label{eq:Mf-dependence-on-zeta-1}
\partial_{\z}\Mf(z,\z) 
= \frac{1}{2}\tf'(\z)\Sch\tf(\z)m(z,\z)^2.
\end{equation}
In particular
\begin{equation} \label{eq:basepoint-stationary-analytic}
\left.\partial_\z\Mf(z,\z)\right|_{\z=z} =0.
\end{equation}
 An additional such result is how the conformal factor $|(\Mf)'(z,\z)|=|\partial_z\Mf(z,\z)|$ depends on $\z$, for which we obtain
\begin{equation} \label{eq:Mf-dependence-on-zeta-2}
\partial_\z\log|\partial_z\Mf(z,\z)| 
= \frac{1}{2}\Sch\tf(\z)m(z,\z).
\end{equation}
Again in particular
\begin{equation} \label{eq:horosphere-raius-stationary-anaylytic}
\left.\partial_\z|(\Mf)'(z,\z)|\right|_{\z=z}=0.
\end{equation}
Equations  \eqref{eq:basepoint-stationary-analytic} and \eqref{eq:horosphere-raius-stationary-anaylytic}
have counterparts in the harmonic case.

Starting with \eqref{eq:M-formula},
\begin{equation} \label{eq:M_xi}
\begin{split}
\partial_\xi\Mf(z,\z) = \partial_\xi\tf(\z) &+ \Re\{\partial_\xi m(z,\z)\} \partial_\xi\tf(\z) + \Re\{m(z,\z)\} \partial_{\xi\xi}\tf(\z)\\
 &+ \Im\{\partial_\xi m(z,\z)\}\partial_\eta\tf(\z) + \Im\{m(z,\z)\}\partial_{\xi\eta}\tf(\z),
 \end{split}
\end{equation}
and we calculate that
\begin{equation} \label{eq:deriv-m-xi}
\begin{aligned}
\partial_\xi m(z,\z) 
&= \frac{-1+\partial_\xi\partial_\z\sigma(\z)(z-\z)^2}{(1-\partial_\z\sigma(\z)(z-\z))^2}.
\end{aligned}
\end{equation}
Similarly,
\begin{equation} \label{eq:deriv-m-eta}
\begin{aligned}
\partial_\eta m(z,\z) & 
= \frac{-i+\partial_\eta\partial_\z\sigma(\z)(z-\z)^2}{(1-\partial_\z\sigma(\z)(z-\z))^2}.
\end{aligned}
\end{equation}
To do more with \eqref{eq:M_xi} we need to work with the second derivatives of $\tf$. Recalling that $\mathbf{N}$ is the unit normal to $\Sigma$, we can write
\[
\begin{aligned}
\partial_{\xi\xi}\tf & = \alpha_{11} \mathbf{N} + \beta_{11} \partial_\xi \tf + \gamma_{11}\partial_\eta \tf, \\
\partial_{\xi \eta}\tf & = \alpha_{12} \mathbf{N} + \beta_{12} \partial_\xi \tf + \gamma_{12}\partial_\eta \tf, \\
\partial_{\eta\eta}\tf &= \alpha_{22} \mathbf{N} + \beta_{22}\partial_\xi\tf +\gamma_{22}\partial_{\eta}\tf.
\end{aligned}
\]
The $\alpha_{ij}$
are the components of the second fundamental form.

We find the other coefficients in terms of the derivatives of $\sigma$. For example, starting with $\langle \partial_{\xi\xi}\tf,\partial_\xi\tf\rangle = \beta_{11}e^{2\sigma}$, from the first equation we have also
\[
 \beta_{11}e^{2\sigma} = \langle \partial_{\xi\xi}\tf,\partial_\xi\tf\rangle  = \frac{1}{2}\partial_\xi\langle \partial_\xi\tf,\partial_\xi\tf\rangle = \frac{1}{2}\partial_\xi(e^{2\sigma})= (\partial_\xi\sigma)e^{2\sigma}.
 \]
 Thus
 \[
 \beta_{11} = \partial_\xi\sigma.
 \]
Similar arguments apply to finding the other coefficients, and the final equations are:
\begin{equation} \label{eq:second-partials}
\begin{aligned}
\partial_{\xi\xi}\tf & = \alpha_{11} \mathbf{N} +\partial_\xi\sigma\, \partial_\xi \tf -\partial_\eta\sigma\,\partial_\eta \tf,\\
\partial_{\xi \eta}\tf & = \alpha_{12} \mathbf{N} +\partial_\eta\sigma \,\partial_\xi \tf + \partial_\xi\sigma\,\partial_\eta \tf,\\
\partial_{\eta\eta}\tf &= \alpha_{22} \mathbf{N} -\partial_\xi\sigma\,\partial_\xi\tf +\partial_\eta\sigma\,\partial_{\eta}\tf.
\end{aligned}
\end{equation}
The derivatives and the $\alpha_{ij}$ are to be evaluated at $\z$, and $\mathbf{N}=\mathbf{N}_{\tw}$.

Substituting this into \eqref{eq:M_xi},
\begin{equation} \label{eq:xi-derivartive-1}
\begin{aligned}
\partial_\xi\Mf(z,\z)
 &= \left[\Re\left\{1+\partial_\xi m(z,\z)+\partial_\xi\sigma(\z)m(z,\z)\right\}+\Im\left\{\partial_\eta\sigma(\z)m(z,\z)\right\}\right]\partial_\xi\tf(\z)\\
 &+\left[\Im\{\partial_\xi m(z,\z)+\partial_\xi\sigma(\z)m(z,\z)\}-\Re\{\partial_\eta\sigma(\z)m(z,\z)\}\right]\partial_\eta\tf(\z)\\
 &+\left[\alpha_{11}(\z)\Re\{m(z,\z)\}+\alpha_{12}(\z)\Im\{m(z,\z)\}\right]\mathbf{N}_{\tw}\\
  \end{aligned}
\end{equation}
Now let
\[
C(z,\z) =1+\partial_\xi m(z,\z) +2\partial_\z\sigma(\z)m(z,\z),
\]
so that
\begin{equation} \label{eq:xi-derivative-2}
\begin{aligned}
\partial_\xi\Mf(z,\z)  &= \Re\{C(z,\z)\} \partial_\xi\tf(\z) + \Im\{C(z,\z)\}\partial_\eta\tf(\z)\\
&\hspace{.4in}+ \left[\alpha_{11}(\z)\Re\{m(z,\z)\}+\alpha_{12}(\z)\Im\{m(z,\z)\}\right]\mathbf{N}_{\tw}.
\end{aligned}
\end{equation}
The terms in the expression for $C(z,\z)$ combine to result in
\[
\begin{aligned}
C(z,\z) &= \frac{(\partial_\xi\partial_\z\sigma(\z)-\partial_\z\sigma(\z)^2)(z-\z)^2}{(1-\partial_\z\sigma(\z)(z-\z))^2}\\
&=m(z,\z)^2(\partial_\xi\partial_\z\sigma(\z)-\partial_\z\sigma(\z)^2).
\end{aligned}
\]
We can take this further, for with
$
\partial_\xi = \partial_\z+\partial_{\bar{\z}},
$
and $\partial_\xi\partial_\z\sigma(\z)-\partial_\z\sigma(\z)^2= \partial_{\z\z}\sigma(\z)-\partial_\z\sigma(\z)^2+\partial_{\z\bar{\z}}\sigma(\z)$
we obtain the final form
\begin{equation} \label{eq:C}
\begin{aligned}
C(z,\z) 
&=m(z,\z)^2\left(\frac{1}{2}\Sch\tf(\z)-\frac{1}{4}e^{2\sigma(\z)}K(\tf(\z))\right).
\end{aligned}
\end{equation}

For later use, let 
\begin{equation} \label{eq:v-and-vn}
\begin{aligned}
\mathbf{v}(z,\z) &= \Re\{C(z,\z)\} \partial_\xi\tf(\z) + \Im\{C(z,\z)\}\partial_\eta\tf(\z), \\
v_n(z,\z) &= \alpha_{11}(\z)\Re\{m(z,\z)\}+\alpha_{12}(\z)\Im\{m(z,\z)\},
\end{aligned}
\end{equation}
 exhibiting
\begin{equation}  \label{eq;v-and-N}
\partial_\xi\Mf(z,\z) = \mathbf{v}(z,\z) + v_n(z,\z)\mathbf{N}_\tw
\end{equation}
as resolved into  velocities tangential to and normal to $T_\tw(\Sigma)$. If $\Sigma$ were planar the normal component would not be present. For the tangential component,
\begin{equation} \label{eq:bound-on-v}
\|\mathbf{v}(z,\z)\| \le \frac{1}{2}e^{\sigma(\z)}|m(z,\z)|^2|(|\Sch\tf(\z)|+\frac{1}{2}e^{2\sigma(\z)}|K(\tf(\z))|),
\end{equation}
and for the normal component,
\begin{equation} \label{eq:bound-on-v_n}
|v_n(z,\z)| \le e^{2\sigma(\z)}|m(z,\z)|\sqrt{|K(\tf(\z))|}
\end{equation}
by the Cauchy-Schwarz inequality and using that for minimal surfaces
\begin{equation} \label{eq:alpha-and-K}
\alpha_{11}^2+\alpha_{12}^2=e^{4\sigma}|K|.
\end{equation}

We  record  some corresponding equations for $\partial_\eta\Mf(z,\z)$. The calculations are very similar, but the end result is a little different, namely
\begin{equation} \label{eq:eta-derivative}
\begin{aligned}
\partial_\eta\Mf(z,\z)&=-\Im\{C'(z,\z)\}\partial_\xi\tf(\z)+\Re\{C'(z,\z)\}\partial_\eta\tf(\z) \\
&\hspace{.4in} + \left[\alpha_{12}(\z)\Re\{m(z,\z)\}+\alpha_{22}(\z)\Im\{m(z,\z)\}\right]\mathbf{N}_{\tw},
\end{aligned}
\end{equation}
where
\begin{equation} \label{eq:C'}
C'(z,\z) = m(z,\z)^2\left(\frac{1}{2}\Sch\tf(\z) +\frac{1}{4}e^{2\sigma(\z)}K(\tf(\z))\right).
\end{equation}
Note that the curvature enters with a plus sign.

As consequences of these expressions we have
\begin{equation} \label{eq:base-point-frozen}
\left.\partial_\xi\Mf(z,\z)\right|_{z=\z} = 0, \quad  \left.\partial_\eta\Mf(z,\z)\right|_{z=\z}=0.
\end{equation}
This is the analog to \eqref{eq:basepoint-stationary-analytic}.

The result we seek on the conformal factor $\|D\Mf(z,\z)\|$ analogous to \eqref{eq:horosphere-raius-stationary-anaylytic}
is
\begin{equation} \label{eq:mixed-derivatives-harmonic}
\left.\partial_\xi \|D\Mf(z,\z)\|\right|_{\z=z} = 0, \quad  \left.\partial_\eta\|D\Mf(z,\z)\|\right|_{\z=z}=0.
\end{equation}
 Because $z\mapsto \Mf(z,\z)$ is conformal it suffices to show that
\begin{equation} \label{eq:conformal-factor-stationary}
\left.\partial_\xi\|\partial_x\Mf(z,\z)\|\right|_{\z=z} = 0= \left.\partial_\eta\|\partial_x\Mf(z,\z)\|\right|_{\z=z}, \quad z=x+iy.
\end{equation}
For this
\[
\partial_x\Mf(z.\z) = \Re\{\partial_xm(z,\z)\}\partial_\xi\tf(\z) + \Im\{\partial_xm(z,\z)\}\partial_{\eta}\tf(\z),
\]
whence from \eqref{eq:partial_x-m}
\[
\|\partial_x\Mf(z,\z)\|=|\partial_xm(z,\z)|e^{\sigma(\z)} = \frac{e^{\sigma(\z)}}{|1-\partial_\z\sigma(\z)(z-\z)|^2}.
\]
Next, we consider
\[
\begin{aligned}
\log\|\partial_x\Mf(z,\z)\| &= \sigma(\z) -\log(1-\partial_\z\sigma(\z)(z-\z)) -\log(1-\overline{\partial_\z\sigma(\z)}\,\overline{(z-\z)})\\
&= \sigma(\z)-\log(1-\partial_\z\sigma(\z)(z-\z)) -\log(1-\partial_{\bar{\z}}\sigma(\z)\,\overline{(z-\z)}).
\end{aligned}
\]
To establish \eqref{eq:conformal-factor-stationary} we can work with $\partial_\z$:
\[
\begin{aligned}
\partial_\z\log\|\partial_x\Mf(z,\z)\| 
&= \frac{(\partial_{\z\z}\sigma(\z)-\partial_\z\sigma(\z)^2)(z-\z)}{1-\partial_\z\sigma(\z)(z-\z)} + \frac{\partial_{\z\bar{\z}}\sigma(\z)\overline{(z-\z)}}{1-\partial_{\bar{\z}}\sigma(\z)\,\overline{(z-\z)}}\\
&= \frac{\frac{1}{2}\Sch\tf(\z)(z-\z)}{1-\partial_\z\sigma(\z)(z-\z)} -\frac{\frac{1}{4}e^{2\sigma(\z)}K(\tf(\z))\overline{(z-\z)}}{1-\partial_{\bar{\z}}\sigma(\z)\,\overline{(z-\z)}}\\
&=\frac{1}{2}\Sch\tf(\z)m(z,\z) - \frac{1}{4}e^{2\sigma(\z)}K(\tf(\z))\overline{m(z,\z)}.
\end{aligned}
\]
Equations \eqref{eq:mixed-derivatives-harmonic} follow from $m(\z,\z)=0$.

\subsection{An Application to Horospheres} \label{subsection:horosphere-frozen}
In the next section we will use the equations above in the proof of the quasiconformality of the extension $\E\tf$. One aspect of this is the geometry of horospheres in $\Hyp$ and their images under the mappings $\Mf(p,\zeta)$ for varying $\z$.

It is a  result from hyperbolic geometry that if $H_\z$ is a horosphere in $\Hyp$ of Euclidean radius $a$ and with base point $\z$, and if $M:\Hyp \to \Hyp$ is a M\"obius transformation, then the Euclidean radius $a'$ of the image horosphere $M(H_\z)$, $M(\z) \ne \infty$, is
\begin{equation} \label{eq:change -in-radius}
a'=\|DM(\z)\|a.
\end{equation}

Certainly the analogous formula holds for $\Mf(p,\z)$ mapping a horosphere $H_\z \subset \Hyp$ to a horosphere $H_\tw \subset \Hyp_\tw(\Sigma)$, $\tw = \Mf(\z,\z)$, but more can be said. Fix a horosphere $H_{\z_0} \subset \Hyp$ of Euclidean radius $a$. For a different base point $\z$ the mapping $p\mapsto \Mf(p,\z)$ takes $\mathbb{C}$ to the tangent plane $T_\tw(\Sigma)$, $\tw= \Mf(\z,\z)$ and takes $H_{\z_0}$ to a horosphere $H_{\tw'}\subset \Hyp_\tw(\Sigma)$ based at $\tw'=\Mf(\z_0,\z)$ and of Euclidean radius, say, $a'$. Then $\tw'$ and $a'$ are both functions of $\z=\xi+i\eta$. From \eqref{eq:base-point-frozen} we conclude
\begin{equation} \label{eq:base-point-frozen-2}
\partial_\xi\tw'|_{\z=\z_0}=0,\quad \partial_\eta\tw'|_{\z=\z_0}, 
\end{equation}
while from \eqref{eq:conformal-factor-stationary} and \eqref{eq:change -in-radius} we also have
\begin{equation} \label{eq:conformal-factor-stationary-2}
\partial_\xi a'|_{\z=\z_0}=0,\quad \partial_\eta a'|_{\z=\z_0}=0.
\end{equation}
Put another way, to first order at $\z_0$,
\[
H_{\tw_0} = \Mf(H_{\z_0},\z_0) = \Mf(H_{\z_0},\z) =H_{\tw'}.
\]

\section{Quasiconformality of the Extension $\E\tf$} \label{section:Ef-quasiconformal}

Recall how the extension is defined, in \eqref{eq:Ef}, for points in space:
\[
\E\tf(p) =
\begin{cases}
\Mf(p,\zeta),& \quad p \in C_\zeta,\\
\tf(p),& \quad p \in \partial\D.
\end{cases}
\]
We will establish the existence of a constant $k(\rho)$ such that
\begin{equation} \label{eq:dilatation-bound-1}
\frac{1}{k(\rho)} \le \frac{\max_{\|X\| =1} \|D\E\tf(p)X\|}{\min_{\|X\| =1} \|D\E\tf(p)X\|} \le k(\rho),
\end{equation}
for $p$ in the upper half-space; the arguments and estimates are identical if $p$ is in the lower half-space. Since $\E\tf$ is a homeomorphism of $\overline{\mathbb{R}^3}$ it then follows that $\E\tf$ is $k(\rho)$-quasiconformal everywhere.

A point $p\in \Hyp$ is the intersection of a circle $C_\z$ with a horosphere $H_\z$ in $\Hyp$ that is tangent to $\D$ at $\z$, and
to assess the distortion one can regard $\E\tf$ as acting in directions tangent to and normal to the circles $C_\z$.  As $C_\z$ is orthogonal to $H_\z$ at $p$ the objective is thus to estimate $\|D\E\tf(p)(X)\|$ when a unit vector $X$ is tangent to $C_\z$ at $p$ and when it is tangent to $H_\z$ at $p$.  For this, we add a parameter $t\in \mathbb{R}$ to the circle-horosphere configuration, aiming to adapt to minimal surfaces  the parallel flow in hyperbolic space introduced by Epstein \cite{epstein:reflections} in his study of the classical Ahlfors-Weill extension.

We need a number of notions and formulas from the hyperbolic geometry of the upper half-space $\Hyp$. To begin with, the upper hemisphere over $\D$ in $\Hyp$, denoted $\Sph(0)$ with parameter $t=0$, is the envelope of the family of horospheres $H_\z(0)$, $\z \in \D$, of Euclidean radius
\[
a(\z,0) = \frac{1}{2}(1-|\z|^2).
\]
Starting at the point $p(\z,0) = C_\z \cap H_\z(0)$, follow the hyperbolic geodesic $C_\z$ at unit speed for a time $t$ to the point $p(\z,t)=C_\z\cap H_\z(t)$, where the horosphere $H_\z(t)$ (still based at $\z$) has radius
\begin{equation} \label{eq:horosphere-radius}
a(\z,t) = e^{2t}a(\z,0) = \frac{1}{2}e^{2t}(1-|\z|^2).
\end{equation}
Here $t>0$ moves $p(\z,t)$  upward from $\Sph(0)$ along $C_\z$ and $t<0$ moves $p(\z,t)$ downward from $\Sph(0)$ along $C_\z$. Fixing $t$ and varying $\z$ defines a surface $\Sph(t)$ that is simply a portion of a sphere that intersects the complex plane along  $\partial\D$. It is the envelope of the family of horospheres $H_\z(t)$ and   $p(\z,t)$ is the point of tangency between $\Sph(t)$ and $H_\z(t)$. Varying $t$ as well then gives a family of hyperbolically parallel surfaces in $\Hyp$. For $t <0$ the surface $\Sph(t)$ lies inside $\Sph(0)$ and for $t>0$ it lies outside $\Sph(0)$. The limiting cases as $t \rightarrow \mp \infty$ are, respectively, $\D$ and its exterior.

The mapping $\z \mapsto p(\z,t)$ is a parametrization  of $\Sph(t)$, and one obtains
\begin{equation} \label{eq:p(z,t)}
p(\z,t) = \left(\frac{1+e^{4t}}{1+e^{4t}|\z|^2}\xi, \frac{1+e^{4t}}{1+e^{4t}|\z|^2}\eta, \frac{e^{2t}(1-|\z|^2)}{1+e^{4t}|\z|^2}\right), \quad \z=\xi+i\eta.
\end{equation}
It is an important fact that this is a conformal mapping.  The corresponding conformal metric on $\D$ is
\begin{equation} \label{eq:conformal-factor-p}
\frac{1+e^{4t}}{1+e^{4t}|\z|^2}|d\z|.
\end{equation}

Consider now the  configuration in the image on applying $\E\tf$. The circles $C_\z$ in the bundle $\frak{C}(\D)$ map to corresponding circles $C_{\tw}$, $\tw=\Mf(\z,\z)$, in the bundle $\frak{C}(\Sigma)$, and for each $\z$ we have the one-parameter family of horospheres
\[
H_{\tw}(t) = \Mf(H_\z(t),\z)
\]
in  $\Hyp_{\tw}(\Sigma)$. The circle $C_{\tw}$ intersects each of the horospheres $H_{\tw}(t)$ orthogonally and we write  $\tilde{p}(\z,t) = C_{\tw} \cap H_{\tw}(t)$, so that for each $t$ the surface  \[ \Sigma(t) = \E\tf(\Sph(t))\] is parametrized by
\begin{equation} \label{eq:tilde-p}
\z \mapsto \tilde{p}(\z,t) = \Mf(p(\z,t),\z), \quad \z \in \D.
\end{equation}
\emph{However}, due to the curvature of $\Sigma$ the surface $\Sigma(t)$ need not be the envelope of the horospheres $H_{\tw}(t)$; they need not be tangent to $\Sigma(t)$ at $\tilde{p}(\z,t)$ and the circle $C_{\tw}$ need not be orthogonal to $\Sigma(t)$ there. Put another way, while the derivative of $\E\tf$ in the direction of a circle $C_\z$ will be tangent to the circle $C_\tw$, the derivative of $\E\tf$ in directions tangent to $\mathbb{S}(t)$ need not necessarily be tangent to the corresponding horospheres $H_\tw$. This is the key difference in geometry between our considerations and the case when $\tf$ is analytic and $\Sigma$ is planar as considered by Epstein. It is also the reason that the dilatation of the extension does not turn out to be a clean $(1+\rho)/(1-\rho)$.

Fix a point \[ p_0=p(\z_0,t_0).\] In the direction of $C_{\z_0}$ the extension acts as the fixed M\"obius transformation $\Mf(p,\z_0)$ and we can express the derivative of $\E\tf$ in that direction using the hyperbolic geometry of $\Hyp$ and of $\Hyp_{\tw_0}(\Sigma)$. The calculation of the derivatives of $\E\tf$ at $p_0$ in directions tangent to $\Sph(t_0)$ is equivalent to finding the derivatives of the mapping \eqref{eq:tilde-p} in the $\z$-variable. This is why we need the results in Section \ref{section:BMA-2}, and we are aided further by the fact that $\z \mapsto p(\z,t)$ is a conformal mapping with a known conformal factor. Calculating the derivative in the $\xi$-direction, the quantity we want is
\begin{equation} \label{eq:DM}
D(\Mf)(p(\z,t_0),\z)\partial_\xi p(\z,t_0) + \partial_\xi \Mf(p(\z,t_0),\z)
\end{equation}
evaluated at $\z=\z_0$. The first term contains the contribution from the differential of the mapping $\Mf$ while the second considers the variation of the M\"obius approximations from point to point. There is no essential difference in estimating the derivative in the $\eta$-direction so we consider only \eqref{eq:DM}. This is a consequence of the formulas \eqref{eq:xi-derivative-2} -- \eqref{eq:C'} in Section \ref{section:BMA-2} and \eqref{eq:p(z,t)}, or also because we can as well work with $\tf \circ M$ for any M\"obius transformation $M$ of the disk.

The expression \eqref{eq:DM} represents a vector in the image of $\E\tf$ and, to make use of the geometry, when $\z=\z_0$ we want to resolve it into components tangent to and normal to $H_{\tw_0}(t_0)$ -- this is central to the argument. First note that since $\Mf(p,\z)$ maps $H_\z(t_0)$ to $H_{\tw}(t_0)$, the first term is tangent to $H_{\tw}(t_0)$ at $\E\tf(p(\z,t_0))=\Mf(p(\z,t_0),\z)$. Moreover, because $\Mf(p,\z)$ is a hyperbolic isometry between $\Hyp$ and $\Hyp_{\tw}(\Sigma)$, the size of this first term is determined by $|\partial_\xi p(\z,t_0)|$ together with the heights of $p(\z,t_0)$ above $\mathbb{C}$ and of $\Mf(p(\z,t_0), \z)$ above $T_\tw(\Sigma)$. We will show:
\begin{enumerate}
\item[(A)] The first term in \eqref{eq:DM} is the dominant one when considering the contributions to the component tangent to $H_{\tw}(t_0)$.
\item[(B)] Except for the factor $|\partial_\xi p(\z,t_0)|$, this first term equals in size the derivative of $\E\tf$ in the direction of  the circle $C_\z$.
\item[(C)] The size of the terms in \eqref{eq:DM} normal to $H_{\tw}(t_0)$ are comparable to the derivative of $\E\tf$ in the direction of the circle $C_\z$.
\end{enumerate}

\subsection{Horospheres and Hyperbolic Stereographic Projection} \label{subsection:horospheres-stereographic}

While the first term in \eqref{eq:DM} is relatively straightforward to analyze, the second is not. To do so we will use a variant of stereographic projection based on horospheres and hyperbolic geodesics that is well suited to our geometric  arrangements and enables us to express a point $p\in \Hyp$ using planar data.

For the model case, suppose $p \in \Hyp$ lies on a horosphere $H$  having Euclidean radius $a$ that is based at $0 \in \mathbb{C}$. Let $C$ be the hyperbolic geodesic in $\Hyp$ passing through $p$ with endpoints $0$ and a point $q\in \mathbb{C}$.   To relate $p$ to $q$ we introduce the angle of elevation $\phi$ of the point $p$ as sited from the origin. Let let $\mathbf{e_r}$ be the unit vector in the radial direction in $\mathbb{C}$ and let $\mathbf{N}$ be the upward unit normal to $\mathbb{C}$ in $\Hyp$. See Figure 1 (in profile).
\begin{figure}
\centering
\includegraphics[scale=.7]{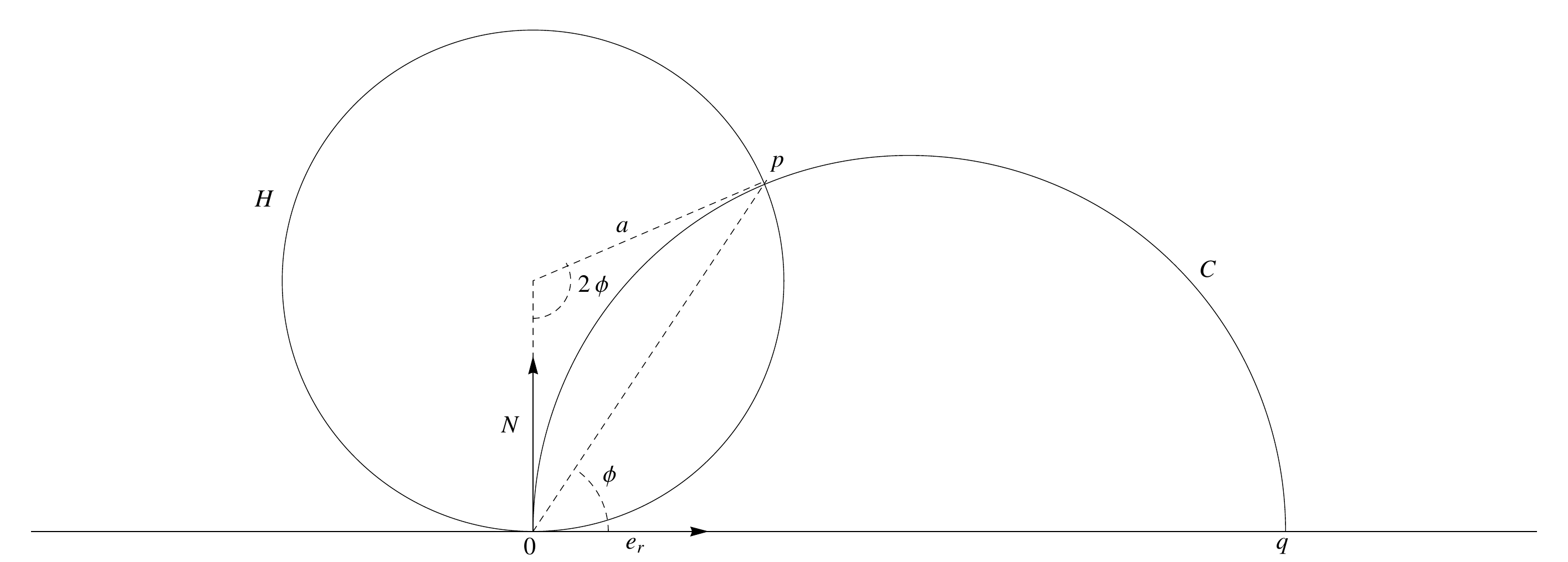}
\caption{Hyperbolic stereographic projection of $q$ to $p$.}
\end{figure}
Then
\[
p= \|q\|\cos\phi((\cos\phi )\mathbf{e_r}+(\sin \phi)\mathbf{N}),
\]
or alternatively
\begin{equation} \label{eq:horosphere-coordinates-2}
p = (\cos^2\!\phi) q + 2a(\sin^2\!\phi)\mathbf{N}.
\end{equation}
Note that $\phi$ depends on $a$. 
Letting $2r = \|q\|$, note also that
\begin{equation} \label{eq:sin-phi-vs-cos-phi}
r\cos\phi = a \sin\phi,
\end{equation}
and that the height of $p$ above the plane is
\begin{equation}
\label{eq:height-of-p}
\mathpzc{h}(p) = 2a\sin^2\!\phi = r \sin2\phi.
\end{equation}

Just as for classical stereographic projection, rays in $\mathbb{C}$ from the origin correspond to meridians on $H$ and circles in $\mathbb{C}$ concentric to the origin correspond to parallels on $H$. It follows that the mapping $q \mapsto p$ is conformal, and it is not difficult to show that the conformal metric on the plane  is $\cos^2 \!\phi|dq|$.

We want to use this to compute $\partial_\xi \Mf(p(\z,t_0),\z)$ at $\z_0$. Again, $p_0=p(\z_0,t_0) = C_{\z_0}\cap H_{\z_0}(t_0)$. As in Section \ref{subsection:horosphere-frozen},  consider a point $\z\in \mathbb{D}$ different from $\z_0$ and the M\"obius transformation $\Mf(p,\z)$. The image $C=\Mf(C_{\z_0},\z)$ is a circle orthogonal to $T_{\tw}(\Sigma)$, $\tw= \Mf(\z,\z)$, passing through $\Mf(\z_0,\z)$  and $\Mf(\z_0^*,\z)$.  The image $H=\Mf(H_{\z_0}(r_0),\z)$ is a horosphere in $\Hyp_{\tw}(\Sigma)$ tangent to $T_{\tw}(\Sigma)$ at $\Mf(\z_0,\z)$; say its radius is ${a}$. We can suppose that $\Mf(\z_0,\z)$ is  the origin of coordinates in $T_{\tw}(\Sigma)$ and apply \eqref{eq:horosphere-coordinates-2} with $C$ and $H$ as above and $q = \Mf(\z_0^*,\z)$ to write
\begin{equation} \label{eq:horosphere-coordinates-image}
\Mf(p_0,\z) = (\cos^2\!{\phi}) \Mf(\z_0^*,\z) + 2{a}(\sin^2\!{\phi})\mathbf{N}.
\end{equation}
Here  $\mathbf{N}$ is the normal to $T_{\tw}(\Sigma)$, and the quantities ${\phi}$, ${a}$ and $\mathbf{N}$ depend on $\z$.

Now recall also from Section \ref{subsection:horosphere-frozen} that to first order at $\z_0$ we have $H_{\tw_0}=\Mf(H_{\z_0},\z_0) = \Mf(H_{\z_0},\z)$. Thus for the purposes of computing the derivative  $\partial_\xi \Mf(p(\z,t_0),\z)$ at $\z_0$ we can regard $\Mf(\z_0^*,\z)$ as varying in the fixed plane $T_{\tw_0}(\Sigma)$ and as being projected to the fixed horosphere $H_{\tw_0}$ along a geodesic $C$ whose one endpoint stays fixed at $\tw_0$ and whose other endpoint is varying in $T_{\tw_0}(\Sigma)$.

\subsection{Components of $\partial_\xi\Mf(p,\z)$} \label{subsection:components-of-partial-M}

Directly from \eqref{eq:horosphere-coordinates-image} we compute 
\[
\begin{aligned}
\partial_\xi \Mf(p_0,\z)&=-2\cos\phi\sin\phi(\partial_\xi\phi)\Mf(\z_0^*,\z)+ (\cos^2\!\phi)\partial_\xi\Mf(\z_0^*,\z)\\
&\hspace{.5in}+2(\partial_\xi a)(\sin^2\!\phi) \mathbf{N} + 4a\sin\phi\cos\phi(\partial_\xi\phi)N +2a(\sin^2\!\phi)\partial_\xi \mathbf{N}.
\end{aligned}
\]
Evaluate this at $\z=\z_0$ using
$
\partial_\xi a|_{\z=\z_0}=0
$
from \eqref{eq:conformal-factor-stationary-2}, and to ease notation write $\phi_0$, $a_0$, $\mathbf{N}_0$  for these quantities at $\z_0$: 
\[
\begin{aligned}
\partial_\xi \Mf(p_0,\z_0)
&= (-2\cos\phi_0\sin\phi_0(\partial_\xi\phi)_0)(\Mf(\z_0^*,\z_0)-2a_0\mathbf{N}_0)\\
&\hspace{.5in}+\cos^2\!\phi_0\,\partial_\xi\Mf(\z_0^*,\z_0)+2a_0\sin^2\!\phi_0\,(\partial_\xi \mathbf{N})_0
\end{aligned}
\]
We invoke \eqref{eq;v-and-N}  for $z=\z_0^*$ to substitute for  $\partial_\xi\Mf(\z_0^*,\z_0)$:
\begin{equation} \label{eq:partial-of-Mf}
\begin{aligned}
\partial_\xi \Mf(p_0,\z_0) &= (-2\cos\phi_0\sin\phi_0(\partial_\xi\phi)_0)(\Mf(\z_0^*,\z_0)-2a_0\mathbf{N}_0)\\
&\hspace{.6in}+\cos^2\!\phi_0\mathbf{v}(\z_0^*,\z_0) + \cos^2\!\phi_0 \,v_n(\z_0^*,\z_0)\mathbf{N}_0 +2a_0(\sin^2\!\phi_0)(\partial_\xi \mathbf{N})_0.
\end{aligned}
\end{equation}

Now isolate the terms
\begin{equation} \label{eq;V-tan}
\begin{aligned}
\mathbf{V}_0&= \mathbf{V}(\z_0)=(-2\cos\phi_0\sin\phi_0(\partial_\xi\phi)_0)(\Mf(\z_0^*,\z_0)-2a_0\mathbf{N}_0)+\cos^2\!\phi_0\mathbf{v}(\z_0^*,\z_0).
\end{aligned}
\end{equation}
$\mathbf{V}_0$ is formed by omitting the terms that are present because $\Sigma$ has curvature. Because of the first-order congruence, as above, if there is no curvature the expression for $\mathbf{V}(\z)$, for $\z=\xi+i\eta_0$ varying in the $\xi$-direction from $\z_0$, is exactly the velocity of a point moving on the fixed horosphere $H_{\tw_0}$ under the hyperbolic stereographic projection of a point moving with velocity $\mathbf{v}(\z^*,\z_0)$ in the fixed plane $T_{\tw_0}(\Sigma)$. Thus in general when $\Sigma$ has curvature, $\mathbf{V}(\z_0)$, having no curvature terms, is tangent to $H_{\tw_0}$ at $\E\tf(p_0)=\Mf(p_0,\z_0)$. Moreover the fact that hyperbolic stereographic projection is conformal with velocities scaling by $\cos^2\phi$ allows us to say, with \eqref{eq:bound-on-v} and $\|\partial_\xi \tf\| = \|\partial_\eta \tf\|=e^\sigma$, that
\begin{equation} \label{eq:bound-on-V}
\begin{aligned}
\| \mathbf{V}(\z_0)\| &= \cos^2\!\phi_0 \,\|\mathbf{v}(\z_0^*,\z_0)\|\\
& \le \frac{1}{2}\cos^2\!\phi_0 \,e^{\sigma(\z_0)}|m(\z_0^*,\z_0)|^2|(|\Sch\tf(\z_0)|+\frac{1}{2}e^{2\sigma(\z_0)}|K(\tf(\z_0))|).
\end{aligned}
\end{equation}

We turn  to the remaining two terms in \eqref{eq:partial-of-Mf},
\begin{equation} \label{eq:vector-w_0}
\mathbf{w}_0=\cos^2\!\phi_0 \,v_n(\z_0^*,\z_0)\mathbf{N}_0 +2a_0(\sin^2\!\phi_0)(\partial_\xi \mathbf{N})_0,
\end{equation}
and seek to write this in the form $\mathbf{W}_0+\mathbf{W}_0^\perp$, for $\mathbf{W}_0=\mathbf{W}(\z_0)$ tangent to and $\mathbf{W}_0^\perp=\mathbf{W}^\perp(\z_0)$ normal to $H_{\tw_0}$ at $\E\tf(p_0)= \Mf(p_0,\z_0)$, respectively.

For this we use polar coordinates $(r,\theta)$ in the plane $T_{\tw_0}(\Sigma)$ with $\tw_0=\Mf(\z_0,\z_0)$ as the origin and we let $\mathbf{e}_r$ and $\mathbf{e}_\theta$ be the orthonormal vectors in the radial and angular directions, respectively.  From the formula \eqref{eq:M-formula} for $\Mf$ we have
\begin{equation} \label{eq:M-w_0}
\Mf(\z_0^*,\z_0) - \tw_0 = \Re\{m(\z_0^*,\z_0)\}\partial_\xi \tf(\z_0)+ \Im\{m(\z_0^*,\z_0)\}\partial_\eta(\z_0),
\end{equation}
and because (as a general fact)
\[
(\partial_\xi\mathbf{N})_0= -(\alpha_{11}(\z_0)\partial_\xi\tf(\z_0) +\alpha_{12}(\z_0)\partial_\eta\tf(\z_0)),
\]
we get
\[
\begin{aligned}
\langle (\partial_\xi \mathbf{N})_0,\Mf(\z_0^*,\z_0)-\tw_0\rangle &= -(\alpha_{11}(\z_0)\Re\{m(\z_0^*,\z_0)\}+\alpha_{12}(\z_0)\Im\{m(\z_0^*,\z_0)\})\\
&= -v_n(\z_0^*,\z_0).
\end{aligned}
\]

Let
\begin{equation} \label{eq:2r_0}
2r_0=\|\Mf(\z_0^*,\z_0) - \tw_0\|.
\end{equation}
 Recalling that $r_0\cos\phi_0 = a_0 \sin\phi_0$ from \eqref{eq:sin-phi-vs-cos-phi}, we can write $\mathbf{w}_0$ as
\[
\begin{aligned}
\mathbf{w}_0&= \cos^2\!\phi_0 \,v_n(\z_0^*,\z_0)\mathbf{N}_0 -2a_0\sin^2\!\phi_0\,v_n(\z_0^*,\z_0)\frac{1}{2r_0}\mathbf{e}_r +\langle \mathbf{w}_0,\mathbf{e}_\theta\rangle \mathbf{e}_\theta\\
&= \cos^2\!\phi_0\,v_n(\z_0^*,\z_0) \mathbf{N}_0 -\sin\phi_0\cos\phi_0 \,v_n(\z_0^*,\z_0)\mathbf{e}_r + \langle \mathbf{w}_0,\mathbf{e}_\theta\rangle \mathbf{e}_\theta.
\end{aligned}
\]
The vector $\langle \mathbf{w}_0,\mathbf{e}_\theta\rangle \mathbf{e}_\theta$ is tangent to $H_{\tw_0}$ at $\E\tf(p_0)$ and points along the latitude through $\E\tf(p_0)$. We can find its magnitude. From \eqref{eq:M-w_0},
\[
2r_0\mathbf{e}_r= \Re\{m(\z_0^*,\z_0)\}\partial_\xi \tf(\z_0)+ \Im\{m(\z_0^*,\z_0)\}\partial_\eta(\z_0),
\]
 and so
\[
2r_0\mathbf{e}_\theta = -\Im\{m(\z_0^*,\z_0)\}\partial_\xi\tf(\z_0) +\Re\{m(\z_0^*,\z_0)\}\partial_\eta\tf(\z_0).
\]
Thus referring back to \eqref{eq:vector-w_0},
\[
\begin{aligned}
\langle \mathbf{w}_0,\mathbf{e}_\theta\rangle &= 2a_0^2\sin^2\!\phi_0\langle(\partial_\xi\mathbf{N})_0, \mathbf{e}_\theta\rangle\\
 &= \frac{2a_0}{2r_0}\sin^2\!\phi_0 \,e^{-\sigma(\z_0)}(\alpha_{11}(\z_0)\Im\{m(\z_0^*,\z_0)\}- \alpha_{12}(\z_0)\Re\{m(\z_0^*,\z_0)\})\\
&= \sin\phi_0\cos\phi_0 \,e^{-\sigma(\z_0)} (\alpha_{11}(\z_0)\Im\{m(\z_0^*,\z_0)\}- \alpha_{12}(\z_0)\Re\{m(\z_0^*,\z_0)\}).
\end{aligned}
\]
The remaining terms in $\mathbf{w}_0$,
\[
\mathbf{u}_0= \cos^2\!\phi_0\,v_n(\z_0^*,\z_0) \mathbf{N}_0 -\sin\phi_0\cos\phi_0\,v_n(\z_0^*,\z_0)\mathbf{e}_r,
\]
must be further resolved to
\[
\mathbf{u_0}= \mathbf{u}_1+\mathbf{W}_0^\perp,
\]
where $\mathbf{u}_1$ is tangent to the longitude through $\E\tf(p_0)$ and  $\mathbf{W}_0^\perp$ is normal to $H_{\tw_0}$ at $\E\tf(p_0)$.

It is easy to check that the vector $\mathbf{u}_0$ makes an angle $\pi/2 - \phi_0$ with the tangent plane to the sphere $H_{\tw_0}$ at $\E\tf(p_0)$, whence
\[
\| \mathbf{u}_1\|= \|\mathbf{u}_0\|\cos\left(\frac{\pi}{2}-\phi_0\right) = \sin\phi_0\cos\phi_0 \,|v_n(\z_0^*,\z_0)|,
\]
while
\[
\|\mathbf{W}_0^\perp\| = \|\mathbf{u}_0\|\sin\left(\frac{\pi}{2}-\phi_0\right)  =\cos^2\!\phi_0\,|v_n(\z_0^*,\z_0)|.
\]
Therefore we can write
\[
\mathbf{w}_0= \mathbf{W}_0+\mathbf{W}_0^\perp,
\]
with
\[
\mathbf{W}_0=\mathbf{u}_1+\langle\mathbf{w}_0,\mathbf{e}_\theta\rangle\mathbf{e}_\theta.
\]
We compute
\[
\begin{aligned}
\| \mathbf{W}_0\|^2 &= \|\mathbf{u}_1\|^2+\langle \mathbf{w}_0,\mathbf{e}_\theta\rangle^2\\
&= \|\mathbf{u}_1\|^2 +\sin^2\!\phi_0\cos^2\!\phi_0\left[\Im\{\alpha_{11}(\z_0)m(\z_0^*,\z_0)\} - \Re\{\alpha_{12}(\z_0)m(\z_0^*,\z_0)\}\right]^2\\
&=\sin^2\!\phi_0\cos^2\!\phi_0[(\Re\{\alpha_{11}(\z_0)m(\z_0^*,\z_0)+\Im\alpha_{12}(\z_0)m(\z_0^*,\z_0))^2\\
&\hspace{.5in} +(\Im\{\alpha_{11}(\z_0)m(\z_0^*,\z_0)\} - \Re\{\alpha_{12}(\z_0)m(\z_0^*,\z_0)\})^2]\\
&= \sin^2\!\phi_0\cos^2\!\phi_0(\alpha_{11}(\z_0)^2+\alpha_{12}(\z_0)^2)(\Re\{m(\z_0^*,\z_0)\}^2+\Im\{m(z_0^*,\z_0)\}^2)\\
&= \sin^2\!\phi_0\cos^2\!\phi_0|m(\z_0^*,\z_0)|^2 e^{4\sigma(\z)}|K(\tf(\z_0)|,
\end{aligned}
\]
where in the last line we have used  $\alpha_{11}^2+\alpha_{12}^2=e^{4\sigma}|K|$, \eqref{eq:alpha-and-K}.
Finally, from the bound on $|v_n|$ in \eqref{eq:bound-on-v_n},
\[
\|\mathbf{W}_0^\perp\| = \cos^2\!\phi_0\, |v_n| \le \cos^2\!\phi_0\,|m(\z_0^*,\z_0)|e^{2\sigma(\z_0)}\sqrt{|K(\tf(\z_0)|}.
\]

Taken together,
\begin{equation} \label{eq:V-W-Wperp-bounds}
\begin{aligned}
\|\mathbf{V}_0\| &\le \frac{1}{2}\cos^2\!\phi_0 \,|m(\z_0^*,\z_0)|^2\,e^{\sigma(\z_0)}(|\Sch\tf(\z_0)|+\frac{1}{2}e^{2\sigma(\z_0)}|K(\tf(\z_0))|),\\
\|\mathbf{W}_0\| &\le \sin\phi_0\cos\phi_0\,|m(\z_0^*,\z_0)| \,e^{2\sigma(\z_0)}\sqrt{|K(\tf(\z_0))|},\\
\|\mathbf{W}_0^\perp\| & \le \cos^2\!\phi_0\,|m(\z_0^*,\z_0)|\,e^{2\sigma(\z_0)}\sqrt{|K(\tf(\z_0)|}.
\end{aligned}
\end{equation}

For the calculations in the next section it will be convenient to write these inequalities 
a little differently. First,
\begin{equation} \label{eq:m-and-r}
|m(\z_0^*,\z_0)| = 2e^{-\sigma(\z_0)}r_0
\end{equation}
from \eqref{eq:M-w_0} and \eqref{eq:2r_0}. Then the bounds for $\|\mathbf{V}_0\|$ and $\|\mathbf{W}_0^\perp\|$ become
\begin{equation} \label{eq:V-Wperp-bounds-alternate}
\begin{aligned}
\|\mathbf{V}_0\| &\le 2r_0^2 \cos^2\!\phi_0 \,e^{-\sigma(\z_0)}(|\Sch\tf(\z_0)|+\frac{1}{2}e^{2\sigma(\z_0)}|K(\tf(\z_0))|),\\
\|\mathbf{W}_0^\perp\| & \le  2r_0\cos^2\!\phi_0\,e^{\sigma(\z_0)}\sqrt{|K(\tf(\z_0)|}.
\end{aligned}
\end{equation}
We also bring in  the radius $a_0$ of the horosphere $H_{\tw_0}$,
\begin{equation} \label{eq:a-0}
a_0 = \frac{1}{2}e^{\sigma(\z_0)}e^{2t_0}(1-|\z_0|^2),
\end{equation}
using \eqref{eq:change -in-radius} and \eqref{eq:horosphere-radius}, and the equation $r_0\cos\phi_0=a_0\sin\phi_0$ from \eqref{eq:sin-phi-vs-cos-phi}. Then the bound for $\|\mathbf{W}_0\|$ is
\begin{equation} \label{eq:W-bound-alternate}
\|\mathbf{W}_0\| \le 4r_0^2\cos^2\!\phi_0\,e^{-2t_0}\frac{\sqrt{|K(\tf(\z_0))|}}{1-|\z_0|^2}.
\end{equation}

We summarize the principal results of this section in a lemma.

\begin{lemma} \label{lemma::components-of-partial-M}
Let $p_0=p(\z_0,t_0) \in H_{\z_0}(t_0)$ and let $\phi_0$ be the angle of elevation of $\E\tf(p_0) \in  H_{\tw_0}=\Mf(H_{\z_0},\z_0)$ measured from $\tw_0$.
Let $2r_0=\|\Mf(\z_0^*,\z_0) - \tw_0\|$.Then
\[
\partial_\xi\Mf(p_0,\z_0) = \mathbf{V}_0+\mathbf{W}_0+\mathbf{W}_0^\perp,
\]
where
$\mathbf{V}_0$ and $\mathbf{W_0}$ are tangent to and $\mathbf{W}_0^\perp$ is normal to  $H_{\tw_0}$ at $\E\tf(p_0)$, with
\begin{equation} \label{eq:V-W-Wperp-bounds-alternate}
\begin{aligned}
\|\mathbf{V}_0\| &\le 2r_0^2 \cos^2\!\phi_0 \,e^{-\sigma(\z_0)}(|\Sch\tf(\z_0)|+\frac{1}{2}e^{2\sigma(\z_0)}|K(\tf(\z_0))|),\\
\|\mathbf{W}_0\| &\le 4r_0^2\cos^2\!\phi_0\,e^{-2t_0}\frac{\sqrt{|K(\tf(\z_0)|}}{1-|\z_0|^2},\\
\|\mathbf{W}_0^\perp\| & \le  2r_0\cos^2\!\phi_0\,e^{\sigma(\z_0)}\sqrt{|K(\tf(\z_0)|}.
\end{aligned}
\end{equation}
\end{lemma}

\subsection{Estimating the Dilatation of $\E\tf$} \label{subsection:dilatation}
The proof of Theorem \ref{theorem:extension} is completed by deriving the bounds
\begin{equation} \label{eq:dilatation-bound-2}
\frac{1}{k(\rho)} \le \frac{\max_{\|X\| =1} \|D\E\tf(p)X\|}{\min_{\|X\| =1} \|D\E\tf(p)X\|} \le k(\rho),
\end{equation}
We continue with the notation as above, in particular working at the fixed point $p_0 =p(\z_0,t_0)=C_{\z_0}\cap H_{\z_0}(t_0)$ and $\tw_0=\tf(\z_0)$.
We estimate $\| D\E\tf(p_0)(X)\|$, $\|X\|=1$ for $X$ tangent to and normal to $C_{\z_0}$ at $p_0$.

\paragraph*{{\bf Case I}: $X$ tangent to $C_{\z_0}$}

Since $\E\tf$ restricted to $C_{\z_0}$ coincides with $\Mf(p,\z_0)$ we have $D\E\tf(p_0)(X) =D\Mf(p_0,\z_0)(X)$, which is tangent to $C_{\tw_0}$ at $\E\tf(p_0)$. The magnitude of $D\Mf(p_0,\z_0)(X)$ can in turn be expressed as the ratio of the heights of $p_0\in \Hyp$ above $\mathbb{C}$ and of $\E\tf(p_0)=\Mf(p_0,\z_0)\in \Hyp_{\tw_0}(\Sigma)$ above $T_{\tw_0}(\Sigma)$, say
\begin{equation} \label{eq:ratio-of-heights}
\|D\E\tf(p_0)X\| = \frac{\mathpzc{h}(\E\tf(p_0))}{\mathpzc{h}(p_0)}.
\end{equation}

\paragraph*{{\bf Case II}: $X$ normal to $C_{\z_0}$} In this case $X$ is a unit vector tangent to the surface $\mathbb{S}(t_0)$ at $p_0$ and we take it to be in the $\partial_\xi$-direction; recall \eqref{eq:DM} and the accompanying discussion.   Letting
\[
\dl_0= \|\partial_\xi p(\z_0,t_0)\|,
\]
we obtain
\[
\begin{aligned}
D\E\tf(p_0)(X) &=D\Mf(p_0,\z_0)(X)+\frac{1}{\dl_0} \partial_\xi\Mf(p_0,\z_0)\\
&= D\Mf(p_0,\z_0)(X)+\frac{1}{\dl_0}( \mathbf{V}_0+\mathbf{W}_0+\mathbf{W}_0^\perp).
\end{aligned}
\]
Hence
\[
\begin{aligned}
\|D\E\tf(p_0)(X)\| &\le \|D\Mf(p_0,\z_0)(X)\| +\frac{1}{\dl_0}(\|\mathbf{V}_0\|+\|\mathbf{W}_0\|+\|\mathbf{W}_0^\perp\|)\\
&= \frac{\mathpzc{h}(\E\tf(p_0))}{\mathpzc{h}(p_0)}+\frac{1}{\dl_0}(\|\mathbf{V}_0\|+\|\mathbf{W}_0\|+\|\mathbf{W}_0^\perp\|).
\end{aligned}
\]
On the other hand, since $D\Mf(p_0,\z_0)(X)$ is tangent to the image horosphere $H_{\tw_0} =\Mf(H_{\z_0},\z_0)$ we also have
\begin{equation} \label{eq:DEf-lower-bound}
\begin{aligned}
\|D\E\tf(p_0)(X)\| &\ge \|D\Mf(p_0,\z_0)(X)\| -\frac{1}{\dl_0}(\|\mathbf{V}_0\|+\|\mathbf{W}_0\|)\\
&= \frac{\mathpzc{h}(\E\tf(p_0))}{\mathpzc{h}(p_0)}-\frac{1}{\dl_0}(\|\mathbf{V}_0\|+\|\mathbf{W}_0\|).
\end{aligned}
\end{equation}

To deduce \eqref{eq:dilatation-bound-2} we thus want to show two things:
\begin{enumerate}
\item There exists a constant $\kappa_1=\kappa_1(\rho) <1$ such that
\begin{equation} \label{eq:dilatation-lower}
\frac{1}{\dl_0}(\|\mathbf{V_0}\| + \|\mathbf{W}_0\|) \le \kappa_1\frac{\mathpzc{h}(\E\tf(p_0))}{\mathpzc{h}(p_0)}.
\end{equation}
\item There exists a constant $\kappa_2=\kappa_2(\rho)<\infty$ such that
\begin{equation} \label{eq:dilatation-upper}
\frac{1}{\dl_0}\|\mathbf{W}_0^\perp\| \le \kappa_2 \frac{\mathpzc{h}(\E\tf(p_0))}{\mathpzc{h}(p_0)}.
\end{equation}
\end{enumerate}

From the formula \eqref{eq:p(z,t)} for $p(\z,t)$ we calculate
\[
\dl_0=\frac{1+e^{4t_0}}{1+e^{4t_0}|\z_0|^2},
\]
and for the height of $p_0$ above $\mathbb{C}$,
\[
\mathpzc{h}(p_0) = e^{2t_0}\frac{1-|\z_0|^2}{1+e^{4t_0}|\z_0|^2}.
\]
In the image, the height of $\E\tf(p_0)$ above $T_{\tw_0}(\Sigma)$ is
\[
\mathpzc{h}(\E\tf(p_0))=2r_0\sin\phi_0\cos\phi_0,
\]
from \eqref{eq:height-of-p}, where  $2r_0 = \|\Mf(\z_0^*,\z_0)-\tw_0\|$, as in \eqref{eq:2r_0}. We also bring in  the radius $a_0$ of the horosphere $H_{\tw_0}$,
\[
a_0 = \frac{1}{2}e^{\sigma(\z_0)}e^{2t_0}(1-|\z_0|^2),
\]
using \eqref{eq:change -in-radius} and \eqref{eq:horosphere-radius}, and the equation $r_0\cos\phi_0=a_0\sin\phi_0$ from \eqref{eq:sin-phi-vs-cos-phi}. We then obtain
\begin{equation} \label{eq:h-and-delta}
\dl_0\frac{\mathpzc{h}(\E\tf(p_0)}{\mathpzc{h}(p_0)} = 4r_0^2\cos^2\!\phi_0\,e^{-\sigma(\z_0)}\frac{1+e^{-4t_0}}{(1-|\z_0|^2)^2}.
\end{equation}

The estimates in \eqref{eq:V-W-Wperp-bounds-alternate} allow us to express the sought for upper bound \eqref{eq:dilatation-lower} as
\begin{equation} \label{eq:dilatation-lower-alternate}
\begin{aligned}
&|\Sch\tf(\z_0)|+\frac{1}{2}e^{2\sigma(\z_0)}|K(\tf(\z_0))|+2 e^{-2t_0}\,e^{\sigma(\z_0)}\frac{\sqrt{|K(\tf(\z_0))|}}{1-|\z_0|^2} \le \kappa_1 \frac{2(1+e^{-4t_0})}{(1-|\z_0|^2)^2}.\\
\end{aligned}
\end{equation}
Let $A_0$ denote the left-hand side of \eqref{eq:dilatation-lower-alternate}. Since $\tf$ satisfies \eqref{eq:AW-condition},
\begin{equation} \label{eq:A_0-bound}
\begin{aligned}
A_0 &\le \frac{2\rho}{(1-|\z_0|^2)^2} -\frac{1}{2}e^{2\sigma(\z_0)}|K(\tf(\z_0)| + 2 e^{-2t_0}\,e^{\sigma(\z_0)}\frac{\sqrt{|K(\tf(\z_0)|}}{1-|\z_0|^2}\\
&= \frac{2\rho}{(1-|\z_0|^2)^2} -\frac{1}{2}B_0^2+2B_0C_0\\
&= \frac{2\rho}{(1-|\z|_0^2)^2}+2C_0^2-\frac{1}{2}(B_0-2C_0)^2,
\end{aligned}
\end{equation}
where we have put
\[
B_0= e^{\sigma(\z_0)}\sqrt{|K(\tf(\z_0))|}, \quad  C_0=\frac{e^{-2t_0}}{(1-|\z_0|^2)}.
\]

The curvature bound \eqref{eq:curvature-bound} implies
\begin{equation} \label{eq:B_0-bound}
B_0 \le \frac{\sqrt{2\rho}}{1-|\z_0|^2} = e^{2t_0}\sqrt{2\rho}\,C_0.
\end{equation}
Let $\epsilon \in (0,1)$, to be determined, and define $\tau_0$ by
\begin{equation} \label{eq:rho-and-tau_0}
e^{2\tau_0}=\frac{\sqrt{2}\epsilon}{\sqrt{\rho}}.
\end{equation}
Suppose first that $t_0 \le \tau_0$. Then \eqref{eq:B_0-bound} implies $B_0 \le 2 \epsilon C_0$. Hence from \eqref{eq:A_0-bound},
\begin{equation} \label{eq:A_0-bound-2}
A_0\le \frac{2\rho}{(1-|\z_0|^2)^2}+2(1-(1-\epsilon)^2)C_0^2 = \frac{2\rho}{(1-|\z_0|^2)^2} + \frac{2\rho_1e^{-4t_0}}{(1-|\z_0|^2)^2},
\end{equation}
setting
\begin{equation} \label{eq:rho_1}
\rho_1 = 1-(1-\epsilon)^2.
\end{equation}
The number $\rho_1$ increases with $\epsilon\in(0,1)$ and
$0<\rho_1<1$.
Hence \eqref{eq:dilatation-lower-alternate}  holds in this case with $\kappa_1=\max\{\rho,\rho_1\}$.

Suppose next that $\tau_0 \le t_0$. Begin with
\[
A_0 \le \frac{2\rho}{(1-|\z_0|^2)^2}+2C_0^2 = \frac{2\rho}{(1-|\z_0|^2)^2}+\frac{2e^{-4t_0}}{(1-|\z_0|^2)^2} \le \frac{2\rho_2(1+e^{-4t_0})}{(1-|\z_0|^2)^2}
\]
provided $\rho_2$ can be chosen so that $\rho+e^{-4t_0} \le \rho_2(1+e^{-4t_0})$, that is so that
\[
e^{-4t_0} \le \frac{\rho_2-\rho}{1-\rho_2}.
\]
Using \eqref{eq:rho-and-tau_0} and $\tau_0 \le t_0$ we are led to the optimal value
\[
\rho_2=\rho\frac{1+2\epsilon^2}{\rho+2\epsilon^2}.
\]
Here $\rho_2$ is decreasing for $\epsilon\in(0,1)$ and
\[
\frac{3\rho}{2+\rho} < \rho_2 <1.
\]
Finally, we choose $\epsilon$ as the unique solution in $(0,1)$ for which $\rho_1=\rho_2$. This common value defines a constant $\kappa_1$ for which \eqref{eq:dilatation-lower-alternate} holds, and thus proves the estimate \eqref{eq:dilatation-lower}.

\subsection*{Remark} Before continuing, we note that a simple approximation for this common value when $\rho \sim 1$ can be obtained by replacing both curves $\rho_1=\rho_1(\epsilon)$, $\rho_2=\rho_2(\epsilon)$ by straight lines, giving
\[
\kappa_1=\frac{2+\rho}{4-\rho}.
\]
It is of some interest to be more precise. The equation
\[
1-(1-\epsilon)^2 =  \rho\frac{1+2\epsilon^2}{\rho+2\epsilon^2}
\]
leads to
\[
\rho= h(\epsilon) = \frac{2\epsilon^3(2-\epsilon)}{3\epsilon^2-2\epsilon+1}.
\]
The function $h(\epsilon)$ is monotonically increasing for $\epsilon \in (0,1)$ with $h(0)=0$ and $h(1)=1$ and has an inverse $\epsilon = k(\rho)$ with the same properties. Then the constant $\kappa_1$ in \eqref{eq:dilatation-lower-alternate} and \eqref{eq:dilatation-lower} is
\[
\kappa_1 =  1-(1-g(\rho))^2.
\]
Moreover, using $h'(1)=0$, $h''(1)=-3$ we see that for $\epsilon \sim 1$,
\[
h(\epsilon) \sim 1-\frac{3}{2}(1-\epsilon)^2,
\]
hence
\[
\kappa_1 \sim 1 - \frac{2}{3}(1-\rho).
\]
This agrees to first order with the approximation $(2+\rho)/(4-\rho)$ for $\epsilon \sim 1$.

\bigskip

Finally, we prove \eqref{eq:dilatation-upper}, which, with \eqref{eq:V-W-Wperp-bounds-alternate} and \eqref{eq:h-and-delta} amounts to
\[
e^{2\sigma(\z_0)} \sqrt{|K(\tf(\z_0)|} \le \kappa_2 r_0\frac{2(1+e^{-4t_0})}{(1-|\z_0|^2)^2},
\]
for an appropriate $\kappa_2$.  But now from \eqref{eq:C-diameter}, in Section \ref{subsection:reflection} on the reflection across $\partial\Sigma$, we have
\[
2r_0 = \frac{e^{\sigma(\z_0)}}{\|\nabla \log\U\tf(\z_0)\|},
\]
and so the inequality we must show reduces to
\[
e^{\sigma(\z_0)}\sqrt{|K(\tf(\z_0)|}\;\| \nabla \log\U\tf(\z_0) \| \le \kappa_2 \frac{(1+e^{-4t_0})}{(1-|\z_0|^2)^2}
\]
Lemma \ref{lemma:grad-u-upper-bound} provides the bound
\[
\| \nabla \log\U\tf(\z_0)\| \le \frac{\sqrt{2}}{1-|\z_0|^2},
\]
so it suffices to find $\kappa_2$ with
\[
\sqrt{2}\,e^{\sigma(\z_0)}\sqrt{|K(\tf(\z_0)|} \le \frac{\kappa_2}{1-|\z_0|^2},
\]
which by \eqref{eq:curvature-bound} will hold for
\[
\kappa_2 =2\sqrt{\rho}.
\]

\bigskip

The estimates \eqref{eq:dilatation-lower} and \eqref{eq:dilatation-upper} together show that
\[
1-\kappa_1 \le \frac{\max_{\|X\| =1} \|D\E\tf(p)X\|}{\min_{\|X\| =1} \|D\E\tf(p)X\|} \le 1+\kappa_1+\kappa_2,
\]
which proves that $\E\tf$ is quasiconformal in $\overline{\mathbb{R}^3}$ with constant
\[
k(\rho) =\frac{1+\kappa_1+\kappa_2}{1-\kappa_1}.
\]
The proof of Theorem \ref{theorem:extension} is complete.

In the classical case, when $\tf$ is analytic and $\Sigma$ is planar, the result generalizes the classical Ahlfors-Weill theorem to provide an extension to space, with the classical dilatation as well.

\begin{corollary} \label{corollary:general-Ahlfors-Weill}
If $f$ is analytic in $\mathbb{D}$ and
\[
|\Sch f(\z)|\le \frac{2\rho}{(1-|z|^2)^2}, \quad \rho <1,
\]
then $\E f$ is $(1+\rho)/(1-\rho)$-quasiconformal.
\end{corollary}

\begin{proof} In this case the curvature is zero. We see from \eqref{eq:dilatation-lower-alternate} that we may take $\kappa_1=\rho$ and from \eqref{eq:V-W-Wperp-bounds-alternate} and \eqref{eq:dilatation-upper} that we may take $\kappa_2=0$.
\end{proof}

\subsection{Quasiconformality of the Reflection $\RR$} \label{subsection:reflection-qc}

Recall from Section \ref{subsection:reflection} the reflection $\RR\colon \Sigma \to \Sigma^*$, defined via the circle bundle $\mathfrak{C}(\Sigma)$ on $\Sigma$ by setting $\RR(\tw) = \tw^*$, where for the circle $C_{\tw}$ we have $C_{\tw} \cap T_{\tw}(\Sigma) = \{\tw,\tw^*\}$. Under $\RR$ the unique critical point of $\U\tf$
is mapped to the point at infinity while the rest of $\Sigma$ is mapped to the surface $\Sigma^*$, and  $\overline{\Sigma} \cup \Sigma^* \cup \{\infty\}$ is a topological sphere.  A limiting case of the preceding estimates allows us to deduce that $\RR$ is quasiconformal.

We must find bounds for $\|D\E\tf(p)X\|$ when $p\in \Sigma$ and $X$ is tangent to $\Sigma$ at $p$. This is a limit of the estimates in Case II, above, as $t_0 \rightarrow -\infty$. For $\kappa_1$, following the analysis when $t_0 \le \tau_0$ starting from \eqref{eq:rho-and-tau_0}, we may let $\tau_0 \rightarrow -\infty$ as well, whence $\epsilon \rightarrow 0$, $\rho_1 \rightarrow 0$ and we can take
\begin{equation} \label{eq:kappa_1=rho}
\kappa_1 = \rho.
\end{equation}
The value of $\kappa_2$ is as before, namely $\kappa_2 = 2\sqrt{\rho}$. Thus

\begin{corollary} \label{corollary:R-quasiconformal}
If $\tf$ satisfies \eqref{eq:AW-condition} with $\rho<1$ then the reflection $\RR$ is $k(\rho)$-quasiconformal with
\[
k(\rho) = \frac{1+\rho +2\sqrt{\rho}}{1-\rho} = \frac{(1+\sqrt{\rho})^2}{1-\rho}.
\]
\end{corollary}
As before, when $\tf$ is analytic and $\Sigma$ is planar the estimates involving the curvature and the second fundamental form do not enter. In that case the bound reduces to the classical $(1+\rho)/(1-\rho)$.

\subsection*{Remark} It is possible to show directly that $\RR$ is quasiconformal using the formula
\begin{equation*} \label{eq:reflection-intrinsic-2}
\RR(\tw) = \tw + 2J(\nabla \log\lambda_\Sigma(\tw)), \quad J(p) = p/\|p\|^2.
\end{equation*}
Very briefly, we can regard $\tw \mapsto R(\tw)$ as a vector field along $\Sigma$ (
not tangent to $\Sigma$) and then compute its covariant derivative $\nabla_X\RR$ in the direction of a vector $X$, $\|X\|=1$, tangent to $\Sigma$. Here $\nabla_X\RR$ is the Euclidean covariant derivative on $\mathbb{R}^3$. 

In terms of the function $\lambda_\Sigma$ on $\Sigma$, \eqref{eq:lambda-Sigma}, and the gradient $\Lambda = \|\nabla\log \lambda_\Sigma\|$ one can show
\[
\frac{4\lambda_\Sigma^2}{\Lambda^2}(1-\rho) \le \| \nabla_X\RR\| \le \frac{4\lambda_\Sigma^2}{\Lambda^2}(1+\sqrt{\rho})^2,
\]
and the corollary follows. The derivation  is interesting, but it requires more preparation.

\section{Quasiconformal Extension of Planar Harmonic mappings} \label{section:planar-extension}

In this section we consider the problem of injectivity and quasiconformal extension for the planar harmonic mapping $f = h + \bar{g}$ under the assumption that its lift $\tf$ satisfies \eqref{eq:AW-condition}. Our method is simply to project from $\overline{\Sigma}\cup\Sigma^*$ to the plane, and the reward is the similarity of the resulting extension of the planar map to the classical Ahlfors-Weill formula applied separately to $h$ and $\bar{g}$.

However, \eqref{eq:AW-condition} alone is not enough. In fact, we are in a situation reminiscent of the original Ahlfors-Weill proof, where we need to know first that the projection is locally injective --  geometrically that $\overline{\Sigma}\cup\Sigma^*$ is locally a graph.
If we assume that $f$ is  locally injective, sense-preserving, and that its dilatation $\omega$ is the square of an analytic function, so $|\omega(\z)| <1$, then at least the surface $\Sigma$ is locally a graph; see \cite {duren:harmonic}. It may exhibit several sheets if $f$ is not injective, and the analysis in Lemma \ref{lemma:Sigma^*-locally-a-graph} below suggests that without a stronger assumption on the dilatation the reflected surface $\Sigma^*$ need not be locally a graph. To address the latter we have the following result.
\begin{lemma} \label{lemma:Sigma^*-locally-a-graph}
Suppose that $f=h+\bar{g}$ is locally injective with dilatation $\omega$ the square of an analytic function, and that $\tf$ satisfies \eqref{eq:AW-condition} for a $\rho <1$. If $\omega$ satisfies
\begin{equation} \label{eq:omega-bounds}
\sup_{\z\in \D}\sqrt{|\omega(\z)|} < \frac{1-\sqrt{\rho}}{1+\sqrt{\rho}}, \quad \z\in \mathbb{D},
\end{equation}
then $\Sigma^*$ is locally a graph.
\end{lemma}

\begin{proof}
Fix a point $\tw_0 = \tf(\z_0)$ on $\Sigma$.  Let $\vartheta$ be the angle of inclination with respect to the vertical of the tangent plane $T_{\tw_0}(\Sigma)$. From the formulas for the components of $\tf$, i.e., the formulas for the Weierstrass-Enneper lift, see \cite{duren:harmonic}, one can show that
\begin{equation} \label{eq:angle-of-inclination}
\tan \vartheta = \frac{2\sqrt{|\omega(\z_0)|}}{1-|\omega(\z_0)|}.
\end{equation}
Now let $X$ be a unit tangent vector to $\Sigma$ at $\tw_0$ and let $(D\E\tf(\z_0)(X))^\top$ and $(D\E\tf(\z_0)(X))^\perp$be respectively the tangential and normal components of $D\E\tf(\z_0)(X)$ Then the angle of inclination of the tangent plane $T_{\tw_0^*}(\Sigma^*)$ to $\Sigma^*$ at $\tw_0^*=\RR(\tw_0)$ is
\[
\vartheta+\tan^{-1}\frac{\|(D\E\tf(\z_0)(X))^\perp\|}{\|(D\E\tf(\z_0)(X))^\top\|}.
\]
The surface $\Sigma^*$ will be locally a graph at $\tw_0^*$ if this angle is $<\pi/2$, and using \eqref{eq:angle-of-inclination} this condition can be written
\begin{equation} \label{eq:local-graph-condition}
 \frac{2\sqrt{|\omega(\z_0)|}}{1-|\omega(\z_0)|}\,\frac{\|(D\E\tf(\z_0)(X))^\perp\|}{\|(D\E\tf(\z_0)(X))^\top\|}<1.
 \end{equation}

We can use limiting cases of previous estimates for $\|D\E\tf\|$ to bound the ratio, namely \eqref{eq:DEf-lower-bound}, \eqref{eq:dilatation-lower} and \eqref{eq:dilatation-upper}, with $\kappa_1=\rho$, from \eqref{eq:kappa_1=rho}, and $\kappa_2=2\sqrt {\rho}$. This results in
\[
 \frac{2\sqrt{|\omega|}}{1-|\omega|}\,\frac{\|(D\E\tf(\z_0)(X))^\perp\|}{\|(D\E\tf(\z_0)(X))^\top\|} \le 4\frac{\sqrt{|\omega(\z_0)|}}{1-\omega(\z_0)}\,\frac{\sqrt{\rho}}{1-\rho}.
 \]
 The right-hand side will be $<1$ precisely when
\[
\sqrt{|\omega(\z_0)|} < \frac{1-\sqrt{\rho}}{1+\sqrt{\rho}}.
\]
\end{proof}

We restate Theorem \ref{theorem:planar-extension} from the introduction, and proceed with the proof.
\begin{thm} \label{theorem:planar-extension-copy}
If $f= h+\bar{g}$ is a locally injective harmonic mapping of $\D$ whose lift $\tf$ satisfies \eqref{eq:AW-condition} for a $\rho<1$ and whose dilatation $\omega$ satisfies \eqref{eq:omega-bounds}, then $f$ is injective and has a quasiconformal extension to $\mathbb{C}$ given by
\begin{equation*} \label{eq:extension-planar-2}
F(\z)=
\begin{cases}
&\hspace{-.15in}f(\z), \; \z\in\overline{\D} \\
&\hspace{-.15in}f(\z^*)+\displaystyle{\frac{(1-|\z^*|^2)h'(\z^*)}{\bar{\z^*}-(1-|\z^*|^2)\partial_z\sigma(\z^*)}+\frac{(1-|\z^*|^2)\overline{g'(\z^*)}}{\z^*-(1-|\z^*|^2)\partial_{\bar{z}}\sigma(\z^*)}}\, , \; \z^*=\frac{1}{\bar{\z}},\;,\z \notin\D.
\end{cases}
\end{equation*}
\end{thm}

\begin{proof}
Without loss of generality we can assume that the unique critical point of $\U\tf$ is the origin. Let $\Pi\colon \mathbb{R}^3 \to \mathbb{C}$ be the projection $\Pi(x_1,x_2,x_3) = x_1+ix_2$. We know that $\overline{\Sigma}\cup \Sigma^*$ is locally a graph over $\mathbb{C}$, and hence the mapping
\[
F(\z)=
\begin{cases}
f(\z),& \quad \z \in \overline{\D},\\
(\Pi \circ \RR)(\tf(\z^*)), & \quad \z \not\in\D
\end{cases}
\]
is locally injective. Note that $F = \Pi \circ \E\tf$ restricted to $\mathbb{C}$.

 Locating the critical point of $\U\tf$ at the origin implies that $F(z) \rightarrow \infty$ as $|z|\rightarrow \infty$. By the monodromy theorem we conclude that $F$ is a homeomorphism of $\mathbb{C}$. In particular, the underlying harmonic mapping $f$ is injective. Moreover, the assumption on $\omega$ implies that the inclinations of both $\Sigma$ and $\Sigma^*$ are bounded away from $\pi/2$, making the projection $\Pi$ quasiconformal. Since by Corollary \ref{corollary:R-quasiconformal} the reflection $\RR$ is quasiconformal, so is $F$.

Let us verify that $F$ has the stated form. From the Weierstrass-Enneper formulas (see again \cite{duren:harmonic}),
 \[
 \begin{aligned}
 \partial_\xi\tf &= (\Re\{h'+g'\},\Im\{h'-g'\},2\Im \{h'\sqrt{\omega}\}),\\
 \partial_\eta \tf& = (-\Im\{h'+g'\},\Re\{h'-g'\}, -2\Re\{h'\sqrt{\omega}\}),
 \end{aligned}
 \]
 from which
 \[
 \Pi(\partial_\xi\tf) = h'+\bar{g}', \quad \Pi(\partial_\eta \tf)= i(h'-\bar{g}').
 \]
 Now recall that the reflection $\RR$ is given in terms of the best M\"obius approximation,  from \eqref{eq:reflection-BMA}, and when $\z $ is outside $\overline{\D}$ we want the projection of
 \[
\tf(\z^*) + \Re\{m(\z^*,\z)\}\partial_\xi(\z^*)+\Im\{m(\z^*,\z)\} \partial_\eta\tf(\z^*).
\]
This is
\[
\begin{split}
f(\z^*) + \Re\{m(\z^*,\z)(h'(\z^*)+&\bar{g}'(\z^*))+i\Im\{m(\z^*,\z)\}(h'(\z^*)-\bar{g}'(\z^*))\\
& = f(\z^*)+m(\z^*,\z)h'+\overline{m(\z^*,\z)}\,\bar{g}'(\z^*),
\end{split}
\]
which is exactly the formula for $F(\z)$.
\end{proof}

We can also conclude that $\overline{\Sigma}\cup\Sigma^*$ is a graph over $\mathbb{C}$.

\bibliographystyle{amsplain}

\providecommand{\bysame}{\leavevmode\hbox to3em{\hrulefill}\thinspace}
\providecommand{\MR}{\relax\ifhmode\unskip\space\fi MR }
\providecommand{\MRhref}[2]{%
  \href{http://www.ams.org/mathscinet-getitem?mr=#1}{#2}
}
\providecommand{\href}[2]{#2}

\end{document}